\documentclass[12pt]{article}
\usepackage{amsmath}
\usepackage{amssymb}
\usepackage{amsthm}
\usepackage{amsfonts}
\usepackage{graphicx}
\usepackage{algorithm}
\usepackage{algpseudocode}
\usepackage[usenames,dvipsnames]{xcolor}
\usepackage[listings, skins]{tcolorbox}
\usepackage{tabularx,siunitx}
\usepackage{subcaption}
\usetikzlibrary{positioning,arrows.meta}

\colorlet{Accent}{Black}

\usepackage{booktabs} % Provides commands for professional-looking tables (\toprule, \midrule, \bottomrule, \addlinespace)
\usepackage[margin=1in]{geometry} % Sets document margins, improving readability
\usepackage{enumitem} % Allows customization of list environments (e.g., changing bullet style)

\usepackage[margin=1in]{geometry} % Set margins
\usepackage{mathtools} % For \coloneqq etc.
\usepackage{bm} % For bold math symbols (\boldsymbol)
\usepackage{hyperref} % For hyperlinks (good practice)

\hypersetup{
    colorlinks=blue,
    linkcolor=blue,
    filecolor=magenta,      
    urlcolor=cyan,
    pdftitle={Operational Levers},
    pdfpagemode=FullScreen,
}

\DeclareTextFontCommand{\textbf}{\color{Accent}\sffamily\bfseries}
\makeatletter
\let\oldbfseries\bfseries
\renewcommand{\bfseries}{\sffamily\color{Accent}\oldbfseries}
\makeatother

\usepackage{enumitem}            % powerful list‑customisation package

\setlist[enumerate]{%
  font=\sffamily\bfseries\color{Accent},   % style for the whole item
  label=\arabic*.                        % (optional) keep the usual numeric label
}
\makeatletter
\let\old@maketitle\@maketitle

\renewcommand{\@maketitle}{%
  {\sffamily\color{Accent}%
    \newpage% (optional – title is usually on its own page)
    \null\vskip 2em% vertical spacing before the title
    \begin{center}%
      {\LARGE \bfseries \@title \par}%
      \vskip 1.5em%
      {\large
        \lineskip .5em%
        \begin{tabular}[t]{c}%
          \@author
        \end{tabular}\par}%
      \vskip 1em%
      {\large \@date \par}%
    \end{center}%
    \par\vskip 1.5em}%
}
\makeatother
\usepackage{titlesec}

\titleformat{\section}
  {\normalfont\normalsize\bfseries}{\thesection}{1em}{\MakeUppercase}
\titleformat{\subsection}
  {\normalfont\normalsize\bfseries}{\thesubsection}{1em}{}
\titleformat{\subsubsection}
  {\normalfont\normalsize\sffamily\color{Accent}\itshape}{\thesubsubsection}{1em}{}

\makeatletter
\renewcommand{\tagform@}[1]{%
  \maketag@@@{\color{Accent}\sffamily(#1)}%
}
\makeatother

\usepackage{fancyhdr}
\renewcommand{\headrulewidth}{0.5pt} % thickness of header rule
\renewcommand{\headrule}{\hbox to\headwidth{\leaders\hrule height \headrulewidth\hfill}} % colored rule

\renewcommand{\headrule}{%
  \vspace{-2.0ex} % negative space to pull the rule up
  \hbox to\headwidth{\leaders\hrule height \headrulewidth\hfill}%
  \vspace{-0.5ex} % optional: pull text closer below if needed
}
\tikzset{every picture/.style={line width=0.75pt}} %set default line width to 0.75pt     

\title{A Sequential Planning Framework for the Operational Reality of Interacting Air Traffic Flow Regulations and Traffic Flow Programs}
\author{\textbf{Thinh Hoang, Daniel Delahaye}\\
École Nationale de l'Aviation Civile\\
Toulouse, France
}
\date{\today}

\theoremstyle{plain}
\newtheorem{proposition}{Proposition}

\theoremstyle{definition}

\newcommand{\kl}[2]{D_{\mathrm{KL}}(#1 \parallel #2)}

\newcommand{\expect}[2]{\mathbb{E}_{#1}\left[#2\right]}

\newcommand{\CC}{\mathcal{C}}

\begin{document}

\maketitle

\begin{abstract}
	Air Traffic Flow Management (ATFM) traffic regulations are being increasingly used as rising demand meets persistent workforce shortages. This operational strain has amplified a critical phenomenon that we call \emph{regulation cascading}: the compounding, non-linear interactions that occur when multiple regulations influence one another in unpredictable ways. As the number and complexity of regulations grow, cascading effects become more pronounced, undermining the network operator's ability to protect sectors reliably.
	
	To address this challenge, we introduce RegulationZero, a sequential planning framework that natively operates in the regulation space, optimizing over ordered sequences of flow-level regulations that remain fully compatible with existing slot-allocation systems such as CASA and RBS++. At its core, the method employs a hierarchical Monte Carlo Tree Search (MCTS) that first samples congestion hotspots and then selects candidate regulations synthesized by a local proposal engine. Each proposal is evaluated by a fast First-Planned-First-Served (FPFS) allocator to estimate its reward, with these feedbacks guiding the subsequent MCTS exploration.
	
	Experiments on pan-European summer-peak traffic confirm that RegulationZero delivers promising and consistent performance. Compared to a simulated-annealing baseline operating in trajectory space, it achieves markedly higher objective improvements, while delaying fewer flights at shorter horizons and maintaining near-parity efficiency in benefit per delayed flight over longer periods. The framework generalizes well across traffic days with diverse congestion patterns and consistently outperforms a greedy rate-capping policy, which shows clear manifestation of regulation cascading effects. Importantly, RegulationZero preserves FPFS fairness and supports expert knowledge injection, offering a pragmatic and low-disruption pathway toward automation in operations.
\end{abstract}

\section{Introduction}
\subsection{The lagging of current  Traffic Flow Management measures behind operational realities.}
The recent surge in air traffic demand has placed immense strain on the current Air Traffic Flow Management (ATFM) system. In Europe, flight volumes increased by 9.8\% to reach 10.2 million flights in 2023, representing approximately 93\% of 2019 levels, with some regions even surpassing pre-pandemic volumes in 2024 \cite{eurocontrol2024network,eurocontrol2023trends}. Across the Atlantic, the U.S. FAA reported that the National Airspace System handled a record 16.3 million flights in 2023 \cite{whitaker2024state}.%, while in China, the airspace system experienced a strong rebound as demand recovered following COVID-19 \cite{caac2024statistical, govcn2025aviation}.

Beyond rising demand, workforce constraints have recently emerged as a critical bottleneck, particularly in Europe and the United States. 
%Eurocontrol’s latest performance review highlighted insufficient air traffic controller (ATC) staffing and slow recruitment as major contributors to capacity shortfalls \cite{aviation24be}.%
Approximately 3 million minutes of en-route ATFM delay in 2023 were attributed solely to controller strikes \cite{eurocontrol2024prr}. In the U.S., FAA data indicate that most air traffic facilities remain understaffed, with the National Airspace System operating at only 72\% of the required number of fully certified controllers \cite{kelly2025understaffed}. % This shortage prompted authorities to impose capacity caps, resulting in significant delays. The FAA reported a 45\% increase in delays and as many as 40,000 flights affected at airports serving the New York area during the summer of 2023 compared to 2022 due to controller shortages \cite{kelly2025understaffed}.

With demand continuing to rise and available capacity lagging behind, the pressure on ATFM positions intensified. %Eurocontrol reported that 2023 marked one of the most challenging years for ATFM performance in two decades, with en-route delays reaching 18.1 million minutes or approximately 301,000 hours, thereby exceeding even the delays recorded in 2019, despite lower traffic volumes that year \cite{aviation24be}. This outcome reflects the increasing reliance on ATFM regulations (or Ground Delay Programs in the U.S.) as a tactical response to congestion. The number of such interventions continues to climb: in 2023, European Air Navigation Service Providers (ANSPs) issued nearly 3,300 zero-rate “scenario” regulations, compared with about 2,500 in 2022 \cite{eurocontrol2024network}.
The above statistics showed that existing capacity management strategies are struggling to match demand, therefore calling for a better ATM solution today, not in the future.
\subsection{The Challenges of Regulation Planning \& Regulation Cascading}
\subsubsection{The Operation of Delay Programs}
The current DCB tasks are chiefly realized world-wide by the slot allocation algorithms, such as Computer Assisted Slot Allocation (CASA) algorithm by Eurocontrol in Europe, or Ration-by-Schedule and Compression (RBS++) in the U.S.

In the European context, FMPs continuously monitor demand and capacity using Traffic Volumes (TFVs) built around a reference location and associated flows, and provide the NM with monitoring values, TFV/flow definitions and regulation proposals. Once NM agrees, it activates a regulation that explicitly contains the validity window, a description of the location/reference, and the entering flow rate. When a regulation is activated, NM’s CASA builds a slot list at the declared rate and allocates CTOTs to the affected flights on a First-Planned-First-Served (FPFS) basis (Figure \ref{fig:ectl_workflow}) \cite{Eurocontrol_ATFCM_2018}. Similarly, in the U.S., ATCSCC sets the program’s time window, rate and flight scope/filters, and TFMS applies RBS++ to allocate take-off slots/EDCTs \cite{FAA_FSM_Users_Guide_2016}.

\begin{figure}
    \centering

    \begin{tikzpicture}[x=0.75pt,y=0.75pt,yscale=-1,xscale=1]
%uncomment if require: \path (0,330); %set diagram left start at 0, and has height of 330

%Shape: Rectangle [id:dp37567925667059265] 
\draw   (130,250) -- (190,250) -- (190,292) -- (130,292) -- cycle ;
%Shape: Rectangle [id:dp4855547022266907] 
\draw   (130,130) -- (348,130) -- (348,172) -- (130,172) -- cycle ;
%Shape: Rectangle [id:dp5746209387413925] 
\draw   (210,250) -- (270,250) -- (270,292) -- (210,292) -- cycle ;
%Shape: Rectangle [id:dp40197600403852984] 
\draw   (290,250) -- (350,250) -- (350,292) -- (290,292) -- cycle ;
%Straight Lines [id:da2590292984694511] 
\draw    (160,250) -- (236.6,171.43) ;
\draw [shift={(238,170)}, rotate = 134.27] [color={rgb, 255:red, 0; green, 0; blue, 0 }  ][line width=0.75]    (10.93,-3.29) .. controls (6.95,-1.4) and (3.31,-0.3) .. (0,0) .. controls (3.31,0.3) and (6.95,1.4) .. (10.93,3.29)   ;
%Straight Lines [id:da5855516848366144] 
\draw    (240,250) -- (238.05,172) ;
\draw [shift={(238,170)}, rotate = 88.57] [color={rgb, 255:red, 0; green, 0; blue, 0 }  ][line width=0.75]    (10.93,-3.29) .. controls (6.95,-1.4) and (3.31,-0.3) .. (0,0) .. controls (3.31,0.3) and (6.95,1.4) .. (10.93,3.29)   ;
%Straight Lines [id:da13189073498868553] 
\draw    (328,250) -- (239.49,171.33) ;
\draw [shift={(238,170)}, rotate = 41.63] [color={rgb, 255:red, 0; green, 0; blue, 0 }  ][line width=0.75]    (10.93,-3.29) .. controls (6.95,-1.4) and (3.31,-0.3) .. (0,0) .. controls (3.31,0.3) and (6.95,1.4) .. (10.93,3.29)   ;
%Shape: Rectangle [id:dp8500637919470575] 
\draw   (132,30) -- (350,30) -- (350,72) -- (132,72) -- cycle ;
%Shape: Rectangle [id:dp8242676778673109] 
\draw   (60,20) -- (380,20) -- (380,180) -- (60,180) -- cycle ;
%Shape: Rectangle [id:dp48167085457477377] 
\draw   (70,240) -- (380,240) -- (380,310) -- (70,310) -- cycle ;
%Shape: Rectangle [id:dp9344819027947989] 
\draw   (390,138) -- (500,138) -- (500,180) -- (390,180) -- cycle ;
%Shape: Rectangle [id:dp5357297506087922] 
\draw   (510,138) -- (620,138) -- (620,180) -- (510,180) -- cycle ;
%Straight Lines [id:da1643124428787487] 
\draw    (450,50) -- (450,138) ;
\draw [shift={(450,140)}, rotate = 270] [color={rgb, 255:red, 0; green, 0; blue, 0 }  ][line width=0.75]    (10.93,-3.29) .. controls (6.95,-1.4) and (3.31,-0.3) .. (0,0) .. controls (3.31,0.3) and (6.95,1.4) .. (10.93,3.29)   ;
%Straight Lines [id:da6336090771856437] 
\draw    (570,50) -- (570,138) ;
\draw [shift={(570,140)}, rotate = 270] [color={rgb, 255:red, 0; green, 0; blue, 0 }  ][line width=0.75]    (10.93,-3.29) .. controls (6.95,-1.4) and (3.31,-0.3) .. (0,0) .. controls (3.31,0.3) and (6.95,1.4) .. (10.93,3.29)   ;
%Straight Lines [id:da8742661114854913] 
\draw    (380,50) -- (570,50) ;
%Straight Lines [id:da07101532778904285] 
\draw    (240,130) -- (240,72) ;
\draw [shift={(240,70)}, rotate = 90] [color={rgb, 255:red, 0; green, 0; blue, 0 }  ][line width=0.75]    (10.93,-3.29) .. controls (6.95,-1.4) and (3.31,-0.3) .. (0,0) .. controls (3.31,0.3) and (6.95,1.4) .. (10.93,3.29)   ;

% Text Node
\draw (135,262) node [anchor=north west][inner sep=0.75pt]   [align=left] {FMP 1};
% Text Node
\draw (169,142) node [anchor=north west][inner sep=0.75pt]   [align=left] {Network Manager};
% Text Node
\draw (215,262) node [anchor=north west][inner sep=0.75pt]   [align=left] {FMP 2};
% Text Node
\draw (295,262) node [anchor=north west][inner sep=0.75pt]   [align=left] {FMP 3};
% Text Node
\draw (96,201) node [anchor=north west][inner sep=0.75pt]   [align=left] {Regulations};
% Text Node
\draw (190,42) node [anchor=north west][inner sep=0.75pt]   [align=left] {CASA Algorithm};
% Text Node
\draw (71,160) node [anchor=north west][inner sep=0.75pt]  [rotate=-270] [align=left] {EUROCONTROL};
% Text Node
\draw (73,298) node [anchor=north west][inner sep=0.75pt]  [rotate=-270] [align=left] {ACCs\\Airports};
% Text Node
\draw (481,62) node [anchor=north west][inner sep=0.75pt]   [align=left] {Delays};
% Text Node
\draw (418,151) node [anchor=north west][inner sep=0.75pt]   [align=left] {Flight 1};
% Text Node
\draw (538,151) node [anchor=north west][inner sep=0.75pt]   [align=left] {Flight 2};

\end{tikzpicture}
    
    \caption{The general Flow Regulation Workflow in the European Context.}
    \label{fig:ectl_workflow}
\end{figure}
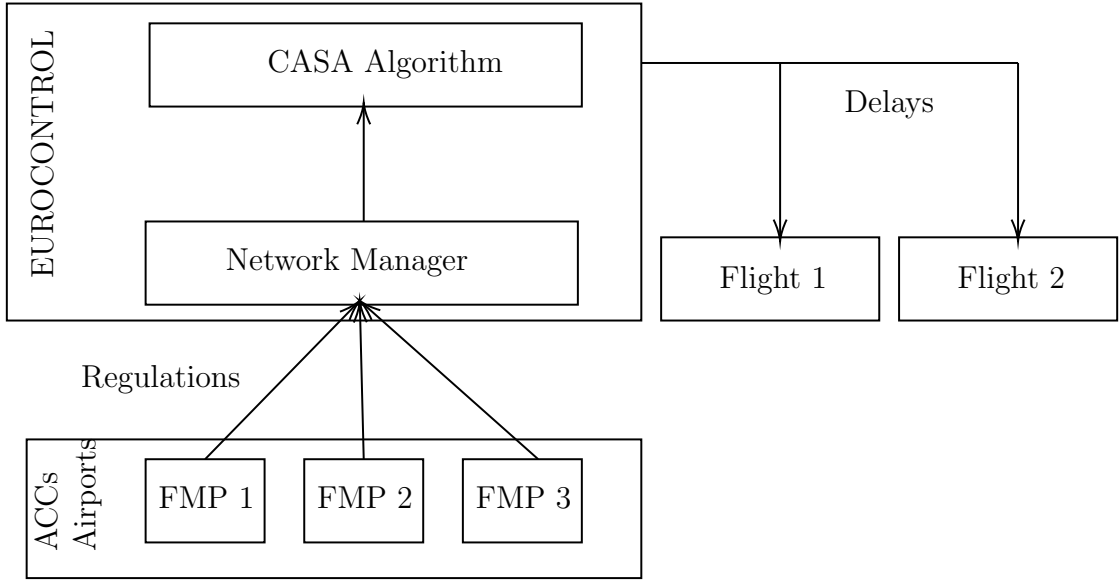

Designing effective flow management regulations is a complex, experience-driven task, especially under uncertain conditions like adverse weather. FMP staff must anticipate and mitigate demand-capacity imbalances hours in advance, often based on imperfect forecasts. In practice this relies heavily on veteran traffic managers’ expertise to interpret evolving situations and coordinate interventions. In 2023, Eurocontrol noted a sharp increase in short-notice adjustments by FMPs, with the frequency of rapid regulation updates doubling compared to 2022. These frequent, reactive changes rendered the daily ATFM plan highly dynamic, demanding constant human oversight and introducing greater unpredictability for airlines.

\subsubsection{The Complex Interactions between Regulations}

Even with perfect traffic forecasts, the fragmented nature of regulation proposals by individual FMP positions at each ACC could result in a set of regulations that may conflict with each other, leading to degraded efficiency and increase Collaborative Decision Making (CDM) efforts. For instance, during periods of convective weather over Central Europe, adjacent FMPs such as those in Munich, Vienna, and Zurich may independently impose flow regulations to protect their sectors. When implemented without cross-border coordination, these local measures can create overlapping or contradictory restrictions along shared traffic flows: one regulation may cancel another's flows, leading to unpredictable effects and increased workload for the Network Manager to reconcile competing constraints in real time.

In this paper, we refer to this phenomenon as \textit{regulation cascading} (RC), which is one of the prominent sources of unpredictability in the network-level regulation planning. Addressing this challenge requires reframing the ATFM DCB problem within a conceptual space distinct from conventional flight-centric paradigms like Trajectory-Based Operations (TBO). We refer to these alternative domains as the regulation space and the trajectory space, respectively. 

Figure~\ref{fig:compare_reg_spaces} contrasts the workflow structures of representative slot allocation algorithms typically found in the literature (operating in the Trajectory Solution Space) with those classified as Regulation Planners (aligned with the Regulation Solution Space). Table \ref{tab:rs_vs_tbo} outlines the key distinctions between the two in details. As is evident from this comparison, searching within the regulation space aligns more  with the structure of a \textit{Markov Decision Process} (MDP) rather than a traditional static optimization problem, since each regulatory action taken induces a transition in the system’s state, such as dumping traffic on another area and thereby, changing the flow components. This sequential decision-making formulation introduces a substantially higher level of complexity: the agent must contend with temporal dependencies and delayed effects of regulations. In essence, the challenge lies not only in identifying an optimal regulation at a given instant, but in learning an adaptive policy that balances short-term overload relief, with long-term network stability.

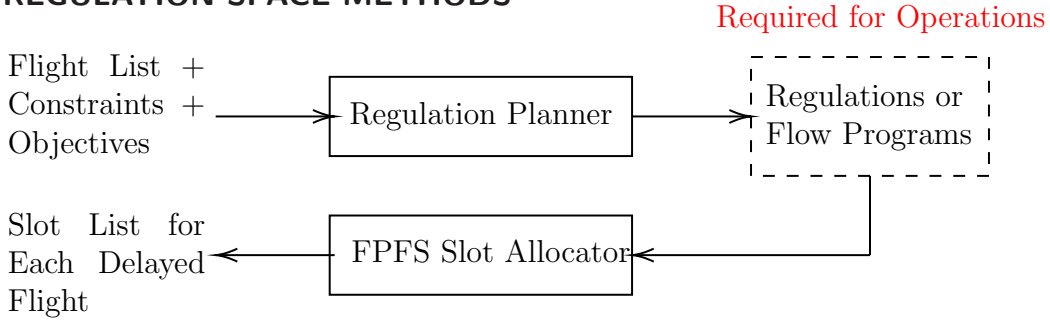
\begin{figure}
    \centering
    \begin{tikzpicture}[x=0.75pt,y=0.75pt,yscale=-1,xscale=1]
%uncomment if require: \path (0,409); %set diagram left start at 0, and has height of 409

%Shape: Rectangle [id:dp27440401099191214] 
\draw   (247.5,70) -- (430,70) -- (430,110) -- (247.5,110) -- cycle ;
%Straight Lines [id:da708895217617869] 
\draw    (190,90) -- (248,90) ;
\draw [shift={(250,90)}, rotate = 180] [color={rgb, 255:red, 0; green, 0; blue, 0 }  ][line width=0.75]    (10.93,-3.29) .. controls (6.95,-1.4) and (3.31,-0.3) .. (0,0) .. controls (3.31,0.3) and (6.95,1.4) .. (10.93,3.29)   ;
%Straight Lines [id:da6688380030281879] 
\draw    (430,90) -- (488,90) ;
\draw [shift={(490,90)}, rotate = 180] [color={rgb, 255:red, 0; green, 0; blue, 0 }  ][line width=0.75]    (10.93,-3.29) .. controls (6.95,-1.4) and (3.31,-0.3) .. (0,0) .. controls (3.31,0.3) and (6.95,1.4) .. (10.93,3.29)   ;
%Shape: Rectangle [id:dp777710065700816] 
\draw   (247.5,230) -- (400,230) -- (400,270) -- (247.5,270) -- cycle ;
%Straight Lines [id:da016414211612456375] 
\draw    (190,250) -- (248,250) ;
\draw [shift={(250,250)}, rotate = 180] [color={rgb, 255:red, 0; green, 0; blue, 0 }  ][line width=0.75]    (10.93,-3.29) .. controls (6.95,-1.4) and (3.31,-0.3) .. (0,0) .. controls (3.31,0.3) and (6.95,1.4) .. (10.93,3.29)   ;
%Straight Lines [id:da39512408738646465] 
\draw    (400,250) -- (458,250) ;
\draw [shift={(460,250)}, rotate = 180] [color={rgb, 255:red, 0; green, 0; blue, 0 }  ][line width=0.75]    (10.93,-3.29) .. controls (6.95,-1.4) and (3.31,-0.3) .. (0,0) .. controls (3.31,0.3) and (6.95,1.4) .. (10.93,3.29)   ;
%Shape: Rectangle [id:dp5619232066498631] 
\draw   (247.5,300) -- (400,300) -- (400,340) -- (247.5,340) -- cycle ;
%Straight Lines [id:da3541522659906784] 
\draw    (520,320) -- (402,320) ;
\draw [shift={(400,320)}, rotate = 360] [color={rgb, 255:red, 0; green, 0; blue, 0 }  ][line width=0.75]    (10.93,-3.29) .. controls (6.95,-1.4) and (3.31,-0.3) .. (0,0) .. controls (3.31,0.3) and (6.95,1.4) .. (10.93,3.29)   ;
%Straight Lines [id:da7737285634159184] 
\draw    (520,280) -- (520,320) ;
%Straight Lines [id:da04956142737877467] 
\draw    (250,320) -- (192,320) ;
\draw [shift={(190,320)}, rotate = 360] [color={rgb, 255:red, 0; green, 0; blue, 0 }  ][line width=0.75]    (10.93,-3.29) .. controls (6.95,-1.4) and (3.31,-0.3) .. (0,0) .. controls (3.31,0.3) and (6.95,1.4) .. (10.93,3.29)   ;
%Flowchart: Process [id:dp8376416823374807] 
\draw  [dash pattern={on 4.5pt off 4.5pt}] (460,220) -- (580,220) -- (580,280) -- (460,280) -- cycle ;

% Text Node
\draw (256,82) node [anchor=north west][inner sep=0.75pt]   [align=left] {Slot Allocation Algorithm};
% Text Node
\draw (135,95) node   [align=left] {\begin{minipage}[lt]{74.8pt}\setlength\topsep{0pt}
Flight List + Constraints + Objectives
\end{minipage}};
% Text Node
\draw (555,95) node   [align=left] {\begin{minipage}[lt]{74.8pt}\setlength\topsep{0pt}
Slot List for Each Delayed Flight
\end{minipage}};
% Text Node
\draw (81,32) node [anchor=north west][inner sep=0.75pt]   [align=left] {\textbf{TRAJECTORY SOLUTION SPACE METHODS}};
% Text Node
\draw (81,181) node [anchor=north west][inner sep=0.75pt]   [align=left] {\textbf{REGULATION SPACE METHODS}};
% Text Node
\draw (135,245) node   [align=left] {\begin{minipage}[lt]{74.8pt}\setlength\topsep{0pt}
Flight List + Constraints + Objectives
\end{minipage}};
% Text Node
\draw (256,241) node [anchor=north west][inner sep=0.75pt]   [align=left] {Regulation Planner};
% Text Node
\draw (466,232) node [anchor=north west][inner sep=0.75pt]   [align=left] {Regulations or\\Flow Programs};
% Text Node
\draw (257,311) node [anchor=north west][inner sep=0.75pt]   [align=left] {FPFS Slot Allocator};
% Text Node
\draw (135,325) node   [align=left] {\begin{minipage}[lt]{74.8pt}\setlength\topsep{0pt}
Slot List for Each Delayed Flight
\end{minipage}};
% Text Node
\draw (441,192) node [anchor=north west][inner sep=0.75pt]  [color={rgb, 255:red, 255; green, 0; blue, 0 }  ,opacity=1 ] [align=left] {Required for Operations};

\end{tikzpicture}
    \caption{Comparison between the optimization pathways of slot allocation algorithms and those of regulation planning processes.}
    \label{fig:compare_reg_spaces}
\end{figure}

\begin{table}[h]
  \centering
  \caption{Comparisons between the TBO and the Regulation Solution Space.} \label{tab:rs_vs_tbo}
  \begin{tabularx}{\linewidth}{p{50pt} X X}
    \toprule
    {\textbf{Aspect}} & {\textbf{Trajectory Space}} & {\textbf{Regulation Space}} \\
    \midrule
    N. variables & \textbf{Fixed:} Could be realized with a fixed binary vector of delay values whose length is the total number of flights. & \textbf{Varying:} The number of regulations to be created could vary from a dozen to hundreds.\\
    \midrule
    Historical dependence & \textbf{Markovian:} The number of changes one can make to any flight's decision variables depend very little on the history (although the change can be accepted or rejected later by the optimizer). & \textbf{Non-Markovian:} Each regulation proposal must base on the selection of a reference location (typically a hotspot) and the traffic flows associated with the reference location. These two information depend on the previous regulations that had been imposed before.\\
    \midrule
    Evaluation & \textbf{Direct: }Objective functions are relatively straightforward to be computed following a decision variable change. & \textbf{Multi-stage with a nested slot allocation algorithm inside:} The regulations ought to be properly parsed, then the CASA algorithm will assign the delays, then the objective function can be calculated.\\
    \midrule 
    Fine-tuning & \textbf{Limited: }Usually a complete re-run, optionally with a convexification term in the objective function will be required if the flight plans change. & \textbf{Liberal: }Many measures are available: tuning the rate, tweaking of flow filters or a complete re-run are all valid options that operators can choose.\\
    \bottomrule
  \end{tabularx}
\end{table}

\subsection{Our Contributions}

To address these challenges, we propose a general framework dubbed RegulationZero based on Monte Carlo Tree Search (MCTS) \cite{silver2018general} designed to search for \textbf{an optimal sequence of flow-level regulations} that resolves DCB problems, in contrast to other flight-based approaches that rely on deterministic delay assignment schemes. Because the traffic situation evolves naturally in the framework, it aims to support operators precisely on complex cross-border regulation interactions that are prone to cause unpredictability which requires a lot of CDM efforts.

The framework is also fully compatible with existing ATFM workflows worldwide, enabling minimal operational disruption, while being performant enough and preserving fairness through the familiar FPFS principle. More importantly, because the search process operates within the same regulation space used by flow manager experts, it enables natural integration of expert intuition to effectively constrain and guide the search domain. Furthermore, the framework exhibits embarrassingly parallel scalability, thus the solution quality could be improved with more compute, without having to spend more time on iterations.

While the end goal is still delaying flights, it is crucial to emphasize that in the current context, \textbf{RegulationZero is not a \textit{slot allocation optimizer}; rather, it operates as a \textit{planning agent} within the broader ATFCM framework}. %It attempts to balance two complementary behaviors: \textit{exploration}, wherein it tests novel regulatory measures to assess their systemic impact, and \textit{exploitation}, wherein it applies established regulations known to yield favorable operational outcomes.

This positions RegulationZero as a dynamic learning system designed to refine regulatory decisions through iterative network responses rather than solely optimizing predefined allocation metrics. \textbf{This is made evident by noting that RegulationZero interacts with the FPFS slot allocation machinery, steering it to yield favorable outcomes rather than doing the optimization by itself. In a strict sense, RegulationZero is a meta-planner for the FPFS slot allocation algorithm.}

\subsection{Paper Organization}

The remainder of this paper is organized as follows. Section 2 provides a concise review of the literature on Traffic Management Initiative (TMI) optimization, with particular emphasis on ground delay strategies. Section 3 introduces the proposed framework, formalized as a cascaded optimization scheme comprising three main components: a MCTS that guides the regulation search process, a flow-level regulation proposal engine addressing each overloaded TFV, and a slot allocation algorithm responsible for computing individual flight delays. Finally, to assess the effectiveness of the proposed methodology, we conduct a large-scale experiment simulating a half day of pan-European air traffic with over 20,000 flights during the 2023 summer peak. The results demonstrate that the framework achieves robust and computationally efficient performance under realistic operational conditions.

\section{Literature Review}
\subsection{Trajectory Space Slot Optimization Methods}
Classical slot allocation in Air Traffic Flow Management (ATFM) originates from ground-holding formulations that optimally assign departure delays under capacity constraints. The foundational single-airport static problem, as proposed by Richetta and Odoni in \cite{RichettaOdoni1993}, established exact optimization principles for minimizing total delay. This formulation was subsequently extended to multi-airport settings, where coordinated delay assignments across interconnected airports were shown in \cite{VranasBertsimasOdoni1994} to yield globally consistent scheduling solutions. Building on these developments, network-flow models later incorporated en-route capacities and routing decisions, anchoring the minimize-total-delay paradigm and demonstrating scalability to realistic traffic networks \cite{BertsimasPatterson1998,BertsimasPatterson2000}. In European operations, the practical counterpart of these optimization concepts has long been institutionalized through the CASA process, implementing FPFS principles codified in the ATFCM manual.

Recent advances in ATFM slot allocation and trajectory space optimization cluster around three interrelated methodological categories:

\subsubsection{Stochastic and MILP formulations under capacity and uncertainty}
Extending beyond single-airport scheduling, recent research integrates airport and fix capacities in multi-airport system (MAS) formulations. Wang \textit{et al.} \cite{Wang2023TRC} and Liu \textit{et al.} \cite{Liu2022JATM,Liu2024TJSASS,Liu2025Aerospace} proposed mixed-integer and robust optimization models that explicitly capture uncertainties in flight and execution times through chance-constrained formulations. These models achieve tractable MILP representations that enhance schedule robustness with limited displacement amplification. Complementarily, He and Pan \cite{HePan2024PLOSONE} coupled capacity-envelope estimation with slot allocation, thereby integrating strategic flow control with empirically grounded throughput constraints.

\subsubsection{Market mechanisms and meta-heuristics}
To accommodate heterogeneous airline valuations and ensure privacy in collaborative decision-making, a number of research has introduced auction-based and credit-driven reallocation mechanisms for ATFM slots. Schuetz \textit{et al.} \cite{Schuetz2022ICAS} and Lee \textit{et al.} \cite{Lee2024JATM} proposed frameworks that integrate market incentives, grandfather rights, and scalable heuristics within privacy-preserving optimization architectures. In parallel, meta-heuristic approaches such as those developed by Tian \textit{et al.} \cite{Tian2021CIE} use evolutionary algorithms to solve large-scale, multi-objective ATFM problems, generating Pareto-efficient schedules that balance delay, cost, and fairness.

For large and tigtly-constrained problems, such as single- and multi-airport settings, enhanced evolutionary and learning-inspired heuristics have delivered high-quality solutions at scale: for example, an estimation-of-distribution metaheuristic coordinating multiple ATFM measures for slot allocation \cite{Tian2021CIE}, and iterated local search / variable neighborhood search tailored to the simultaneous, system-wide slot allocation problem \cite{Pellegrini2011SSRN}. From the airline perspective during severe airspace flow programs, genetic algorithms that encode route and slot decisions have been shown to reduce disruption and cost \cite{Abdelghany2007JATM}. Beyond pure slotting, hybrid metaheuristics designed by \textit{Chaimatanan}, \textit{Delahaye}, and \textit{Mongeau} strategically deconflict trajectories at continental scale jointly adjusting routes and departure times to balance safety and efficiency \cite{Chaimatanan2014CIM}. Complementing these, co-evolutionary genetic algorithms have been applied to ATFM under dynamic capacity, illustrating how cooperative search can handle network-level congestion while respecting capacity constraints and time-slot assignments \cite{Zhang2007ITSC}.

\subsubsection{Network-wide, multi-objective fairness, connectivity, and environmental goals}
Another emerging direction embeds broader operational and environmental criteria into slot allocation. Keskin and Zografos \cite{KeskinZografos2023TRB}, Feng \cite{Feng2023TRD}, and Dalmau \cite{Dalmau2024SIDs} formulated bi-objective and lexicographic optimization models that incorporate fairness across airlines, onward connectivity preservation, and environmental impact mitigation (e.g., noise abatement). These formulations explicitly characterize trade-offs between delay displacement, connectivity protection, and sustainability, contributing to a framework for MAS-aware, data-driven slot allocation that is both robust to uncertainty and sensitive to multi-stakeholder objectives.

Empirical analyses and simulation-based studies further contextualize these approaches. Lau \textit{et al.} \cite{LauBuddeBerlingGollnick2014} evaluated CASA-like heuristics at network scale, benchmarking their performance against optimization-based baselines. Ivanov \textit{et al.} \cite{IvanovEtAl2017} proposed a two-level mixed-integer model that shapes delay distributions to mitigate propagation effects and improve adherence to airport slot assignments. Similarly, Chen et al. \cite{ChenChenSun2017} demonstrated that chance-constrained formulations can incorporate capacity uncertainty while maintaining computational tractability for large-scale instances. Most recently, Berling, Lau, and Gollnick \cite{BerlingLauGollnick2024} integrated a constraint-reconciliation and column-generation allocator into R-NEST, achieving quantifiable reductions in both primary and reactionary delays while respecting operational constraints on the day of execution.

Despite important advances, most prior work frames ATFM as direct, deterministic flight-level slot/delay allocation, which typically demands deep re-engineering, testing, and certification of legacy systems. This flight-centric view is also not well-aligned with practice: human operational expertise typically in regulation space in terms of rates, occupancy, counts; making controller intuition hard to inject into slot-based optimizations, as well as complicating the verification of the regulation plan's soundness against edge cases.

\subsection{Why Regulations still Persist after Decades of Slot Optimization Research}

In essence, current European ATFCM operations prioritize predictability and stability of outcomes over strict computational optimality, due to its specific context of relying on CDM as the central philosophy linking towers, NM, and airline operations. Owing to the inherently volatile nature of upstream data ranging from dynamic flight plan updates and reactive knock-on delays to unpredictable weather phenomena, routine updates such to EOBT and DPI routinely trigger recalculations within the ETFMS. 

While the theoretical foundation of slot optimization is well-established, executing a full recalculation cycle every 15 minutes in response to upstream perturbations would lead to a cascade of delay reassignment messages that must be sent, processed, and validated. This process imposes an exceptionally high communication and coordination burden across all actors. Consequently, controllers are explicitly advised to avoid unnecessary CTOT modifications, especially once a significant number of Slot Allocation Messages have already been disseminated \cite{EUROCONTROL_ATFCM_Manual_2024}. Furthermore, excessive delay reassignment can undermine airline schedule integrity, complicating fleet and crew planning.

In practice, the combination of the FPFS slot allocation and prudent \textit{regulation planning} produces relatively stable flight sequences and predictable CTOTs that are revised only when input conditions or capacity constraints materially change \cite{EUROCONTROL_ATFCM_Manual_2024}. Under this regime, ETFMS autonomously issues adjusted CTOTs when necessary, thereby allowing operational actors to concentrate on higher-level collaborative decision-making and coordination.

\subsection{Regulation Space Optimization Methods}
Despite the pressing operational needs, optimization at the regulation level, on the other hand, are rarely explored. \cite{Dalmau2021Indispensable} observed that not all regulations are equally important, and building upon that, \cite{Dalmau2022Optimal} proposed to construct regulations greedily by capping traffic flow at nominal capacity.
%A subsequent refinement phase used Adaptive Large Neighborhood Search (ALNS) metaheuristics to iteratively remove and rebuild a portion of regulations according to heuristics such as High Delay or Random removal. However, because the regulatory candidates provided to ALNS are generated through greedy heuristics, the approach offers limited control over the rate adjustments or the shaping of individual flows contributing to local overloads. Moreover, the resulting regulation plan was applied simultaneously across the network, giving rise to significant cross-interactions of traffic flows between congested areas. We will show more details in Section 3.1.

\subsection{Positioning of our Paper}
The foregoing review showed a clear gap in existing research and operational practice that this paper seeks to address:

\begin{itemize}
\item In current operations, virtually all ATFM systems worldwide implement slot allocation, whereby delays are churned out by an algorithm following a sets of \textit{regulations} manually designed by expert flow managers.
\item Even forthcoming system evolutions such as the \textit{Targeted CASA} initiative are focused primarily on providing more refined control over traffic selection and rate adjustment, rather than fundamentally replacing the underlying FPFS allocation and regulation mechanisms \cite{EUROCONTROL_NOP_2024_2029}.
\item Existing formulations on \textit{automated regulation design} insufficiently capture two essential dimensions of a regulatory decision: the adjustment of the \textit{regulation rate} and the explicit modeling of \textit{contributing flows}. Consequently, the development of a comprehensive, fully automated regulation planning methodology remains an open challenge.
\item Eurocontrol also noted that interacting flow management measures can yield complex and sometimes counterintuitive effects \cite{EUROCONTROL_ATFCM_Manual_2024}. Therefore, automation of regulation design is a relatively new and under-researched area that could have high-impacts on operational realities.
\end{itemize}

Our paper directly addresses this unfilled operational and methodological gap. We position our contribution as a \textit{decision support framework} that augments existing ATFM operational processes rather than replacing them. Specifically, we develop an optimization-based methodology that automatically constructs regulation plans by jointly determining rate profiles and flow participation under realistic capacity and connectivity constraints. As a result, \textbf{the solutions provided by our framework are fully compatible with both current CASA operations and forthcoming \textit{Targeted CASA} implementations, as the framework accommodates either single-rate or multi-rate specifications across subflows} \cite{EUROCONTROL_NOP_2024_2029}.

\section{Methodology}\label{sec:method}
\subsection{Regulations, Regulation Ordering, and Regulation Cascading}

In European ATFM, a regulation is a flow-management measure that constrains demand to an ATC capacity, and, when activated, triggers the Network Manager’s slot allocator (ETFMS's CASA Algorithm) to sequence affected flights and issue CTOTs to operators and towers. As shown in Figure \ref{fig:regulations_and_flow_programs}, a regulation (or a Traffic Flow Program) typically contains four elements: a reference location (sector/airport/traffic volume), the traffic filtering conditions (the flow to be captured), an active time window, and an allowed entry rate (usually in entries per hour). 

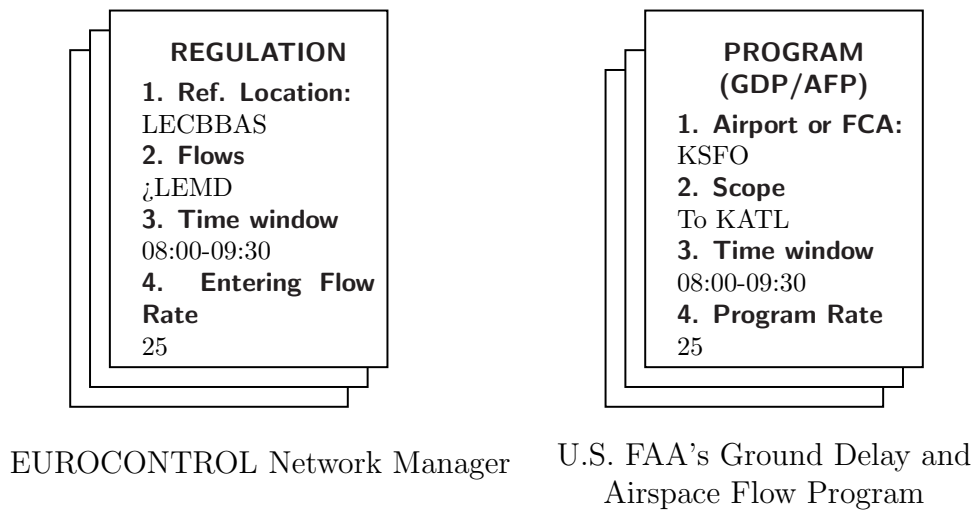
\begin{figure}
    \centering\begin{tikzpicture}[x=0.75pt,y=0.75pt,yscale=-1,xscale=1]

%Shape: Rectangle [id:dp17443929940203262] 
\draw  [fill={rgb, 255:red, 255; green, 255; blue, 255 }  ,fill opacity=1 ] (320,70) -- (460,70) -- (460,240) -- (320,240) -- cycle ;
%Shape: Rectangle [id:dp7471195663095772] 
\draw  [fill={rgb, 255:red, 255; green, 255; blue, 255 }  ,fill opacity=1 ] (50,60) -- (190,60) -- (190,240) -- (50,240) -- cycle ;
%Shape: Rectangle [id:dp9213953489887077] 
\draw  [fill={rgb, 255:red, 255; green, 255; blue, 255 }  ,fill opacity=1 ] (60,50) -- (200,50) -- (200,230) -- (60,230) -- cycle ;
%Shape: Rectangle [id:dp30836331969536157] 
\draw  [fill={rgb, 255:red, 255; green, 255; blue, 255 }  ,fill opacity=1 ] (70,40) -- (210,40) -- (210,220) -- (70,220) -- cycle ;
%Shape: Rectangle [id:dp6426993969771743] 
\draw  [fill={rgb, 255:red, 255; green, 255; blue, 255 }  ,fill opacity=1 ] (330,60) -- (470,60) -- (470,230) -- (330,230) -- cycle ;
%Shape: Rectangle [id:dp7680650702861795] 
\draw  [fill={rgb, 255:red, 255; green, 255; blue, 255 }  ,fill opacity=1 ] (340,40) -- (480,40) -- (480,220) -- (340,220) -- cycle ;

% Text Node
\draw (145,135) node  [font=\footnotesize] [align=left] {\begin{minipage}[lt]{88.4pt}\setlength\topsep{0pt}
\begin{center}
\textbf{{\footnotesize REGULATION}}
\end{center}
{\footnotesize \textbf{1. Ref. Location:}}\\{\footnotesize LECBBAS}\\{\footnotesize \textbf{2. Flows}}\\{\footnotesize >LEMD}\\{\footnotesize \textbf{3. Time window}}\\{\footnotesize 08:00-09:30}\\{\footnotesize \textbf{4. Entering Flow Rate}}\\{\footnotesize 25}
\end{minipage}};
% Text Node
\draw (18,261) node [anchor=north west][inner sep=0.75pt]   [align=left] {EUROCONTROL Network Manager};
% Text Node
\draw (415,135) node  [font=\footnotesize] [align=left] {\begin{minipage}[lt]{88.4pt}\setlength\topsep{0pt}
\begin{center}
{\footnotesize \textbf{PROGRAM (GDP/AFP)}}
\end{center}
{\footnotesize \textbf{1. Airport or FCA:}}\\{\footnotesize KSFO}\\{\footnotesize \textbf{2. Scope}}\\{\footnotesize To KATL}\\{\footnotesize \textbf{3. Time window}}\\{\footnotesize 08:00-09:30}\\{\footnotesize \textbf{4. Program Rate}}\\{\footnotesize 25}
\end{minipage}};
% Text Node
\draw (400,275) node   [align=left] {\begin{minipage}[lt]{176.8pt}\setlength\topsep{0pt}
\begin{center}
U.S. FAA's Ground Delay and \\Airspace Flow Program
\end{center}

\end{minipage}};

\end{tikzpicture}
    \caption{Regulations (EU) and Flow Programs (US) that will serve as input for slot allocation algorithms.} \label{fig:regulations_and_flow_programs}
\end{figure}

The fundamental principle in regulation design is to target traffic volumes (TFVs) exhibiting DCB imbalances. Once identified, the entry rate is capped according to the declared sector capacity that air traffic controllers can safely accommodate. The flow component of the regulation provides operators with more surgical control, enabling targeted interventions on flights contributing to the imbalance while minimizing overall impact and avoiding an excessive number of delay assignment messages \cite{EurocontrolFlowConops2024}.

\subsection{Motivating Example: How the Order of Application Shapes Network Responses}
As discussed above, regulations may interact in nonlinear and sometimes counterintuitive ways. In the example illustrated in Figure~\ref{fig:regulation_order_example}, the sequencing of regulation application can lead to markedly different traffic flow outcomes:  
\begin{itemize}
    \item \textbf{Case 1:} A regulation is first applied to TFV~A, reducing flow~(i) from five to three flights while leaving flow~$x_1$ (comprised of a single flight) unaffected. A subsequent regulation on TFV~C then fully restricts flow~$x_1$ and reduces flow~(ii) from seven to three flights.
    \item \textbf{Case 2:} If the first regulation instead targets TFV~C, it may simultaneously affect flows~(i) and~(ii). Should flow~(ii) be scheduled later than flow~(i), this regulation could entirely hold all flights in flow~(ii) for later release. A subsequent regulation at TFV~A would then further constrain flow~(i) by reducing its entry rate.
\end{itemize}

\textbf{These contrasting outcomes underscore the sensitivity of network behavior to the order in which regulations are implemented.
} In the current operation, this aspect is often overlooked. Regulations are commonly applied simultaneously, and each flight is assigned a delay corresponding to the maximum value among all applicable regulations \cite{EUROCONTROL_ATFCM_Manual_2024}. This observation likely explains the ``surprising effects'' and pronounced volatility frequently reported by operators. Furthermore, the flow patterns considered for a new regulation depend on the set of previously implemented regulations, thereby violating the assumption of a fixed action space, hence poses significant technical challenges for the direct application of conventional metaheuristic or dynamic programming approaches, which typically require an invariant search space.

\begin{figure}
    \centering
    \begin{tikzpicture}[x=0.75pt,y=0.75pt,yscale=-1,xscale=1]
%uncomment if require: \path (0,642); %set diagram left start at 0, and has height of 642

%Shape: Polygon [id:ds17155344184493693] 
\draw  [dash pattern={on 4.5pt off 4.5pt}] (112,100.67) -- (162,130.67) -- (112,210.67) -- (82,190.67) -- (62,130.67) -- (81.67,102) -- cycle ;
%Shape: Polygon [id:ds0446743447594814] 
\draw  [dash pattern={on 4.5pt off 4.5pt}] (112,210.67) -- (162,130.67) -- (226.33,118.67) -- (207.67,206.67) -- cycle ;
%Straight Lines [id:da9029455783081141] 
\draw [color={rgb, 255:red, 255; green, 0; blue, 0 }  ,draw opacity=1 ]   (30.67,202.67) -- (88.33,173.33) -- (127.67,126.67) -- (209.67,82.67) ;
%Straight Lines [id:da6355599554396787] 
\draw [color={rgb, 255:red, 255; green, 0; blue, 0 }  ,draw opacity=1 ]   (127.67,126.67) -- (88.33,173.33) -- (81.67,220) -- (38.33,258) ;
%Shape: Polygon [id:ds07370410543209416] 
\draw  [dash pattern={on 4.5pt off 4.5pt}] (125.67,20) -- (233,52) -- (226.33,118.67) -- (162,130.67) -- (112,100.67) -- cycle ;
%Shape: Rectangle [id:dp4334732927936412] 
\draw   (20,10) -- (280,10) -- (280,270) -- (20,270) -- cycle ;
%Straight Lines [id:da8192788701961259] 
\draw [color={rgb, 255:red, 74; green, 144; blue, 226 }  ,draw opacity=1 ]   (200,90) -- (127.67,126.67) -- (170,220) -- (250,264) ;
%Flowchart: Extract [id:dp31089181534210386] 
\draw  [fill={rgb, 255:red, 0; green, 0; blue, 0 }  ,fill opacity=1 ] (209.67,77.67) -- (214.67,87.67) -- (204.67,87.67) -- cycle ;
%Shape: Polygon [id:ds6500521204483622] 
\draw  [dash pattern={on 4.5pt off 4.5pt}] (402,100.67) -- (452,130.67) -- (402,210.67) -- (372,190.67) -- (352,130.67) -- (371.67,102) -- cycle ;
%Shape: Polygon [id:ds8501933299076837] 
\draw  [dash pattern={on 4.5pt off 4.5pt}] (402,210.67) -- (452,130.67) -- (516.33,118.67) -- (497.67,206.67) -- cycle ;
%Straight Lines [id:da10035153351906145] 
\draw [color={rgb, 255:red, 255; green, 0; blue, 0 }  ,draw opacity=1 ]   (417.67,126.67) -- (378.33,173.33) -- (371.67,220) -- (328.33,258) ;
%Shape: Polygon [id:ds0323302134206781] 
\draw  [dash pattern={on 4.5pt off 4.5pt}] (415.67,20) -- (523,52) -- (516.33,118.67) -- (452,130.67) -- (402,100.67) -- cycle ;
%Shape: Rectangle [id:dp6601944830997593] 
\draw   (310,10) -- (570,10) -- (570,270) -- (310,270) -- cycle ;
%Straight Lines [id:da7765795474587222] 
\draw [color={rgb, 255:red, 74; green, 144; blue, 226 }  ,draw opacity=1 ]   (490,90) -- (417.67,126.67) -- (460,220) -- (540,264) ;
%Flowchart: Extract [id:dp10093944803604782] 
\draw  [fill={rgb, 255:red, 0; green, 0; blue, 0 }  ,fill opacity=1 ] (499.67,77.67) -- (504.67,87.67) -- (494.67,87.67) -- cycle ;
%Shape: Polygon [id:ds43603165590993387] 
\draw  [dash pattern={on 4.5pt off 4.5pt}] (402,409.67) -- (452,439.67) -- (402,519.67) -- (372,499.67) -- (352,439.67) -- (371.67,411) -- cycle ;
%Shape: Polygon [id:ds5800512331060821] 
\draw  [dash pattern={on 4.5pt off 4.5pt}] (402,519.67) -- (452,439.67) -- (516.33,427.67) -- (497.67,515.67) -- cycle ;
%Straight Lines [id:da7737021734662636] 
\draw [color={rgb, 255:red, 255; green, 0; blue, 0 }  ,draw opacity=1 ]   (417.67,435.67) -- (378.33,482.33) -- (371.67,529) -- (328.33,567) ;
%Shape: Polygon [id:ds3201775567584212] 
\draw  [dash pattern={on 4.5pt off 4.5pt}] (415.67,329) -- (523,361) -- (516.33,427.67) -- (452,439.67) -- (402,409.67) -- cycle ;
%Shape: Rectangle [id:dp3948861299749814] 
\draw   (310,319) -- (570,319) -- (570,579) -- (310,579) -- cycle ;
%Flowchart: Extract [id:dp08250785010125505] 
\draw  [fill={rgb, 255:red, 0; green, 0; blue, 0 }  ,fill opacity=1 ] (499.67,386.67) -- (504.67,396.67) -- (494.67,396.67) -- cycle ;
%Straight Lines [id:da6166886485871415] 
\draw [color={rgb, 255:red, 255; green, 0; blue, 0 }  ,draw opacity=1 ]   (320.67,511.67) -- (378.33,482.33) -- (417.67,435.67) -- (499.67,391.67) ;

% Text Node
\draw (34.67,97) node [anchor=north west][inner sep=0.75pt]   [align=left] {TV A};
% Text Node
\draw (225.33,161.67) node [anchor=north west][inner sep=0.75pt]   [align=left] {TV B};
% Text Node
\draw (237.33,73) node [anchor=north west][inner sep=0.75pt]   [align=left] {TV C};
% Text Node
\draw (91,282) node [anchor=north west][inner sep=0.75pt]   [align=left] {Original Situation};
% Text Node
\draw (27,188) node [anchor=north west][inner sep=0.75pt]  [font=\scriptsize] [align=left] {x1};
% Text Node
\draw (31,242) node [anchor=north west][inner sep=0.75pt]  [font=\scriptsize] [align=left] {x5};
% Text Node
\draw (241,242) node [anchor=north west][inner sep=0.75pt]  [font=\scriptsize] [align=left] {x7};
% Text Node
\draw (324.67,97) node [anchor=north west][inner sep=0.75pt]   [align=left] {TV A};
% Text Node
\draw (515.33,161.67) node [anchor=north west][inner sep=0.75pt]   [align=left] {TV B};
% Text Node
\draw (527.33,73) node [anchor=north west][inner sep=0.75pt]   [align=left] {TV C};
% Text Node
\draw (331,282) node [anchor=north west][inner sep=0.75pt]   [align=left] {Regulation at TV A, then at TV C };
% Text Node
\draw (321,242) node [anchor=north west][inner sep=0.75pt]  [font=\scriptsize] [align=left] {x1};
% Text Node
\draw (531,242) node [anchor=north west][inner sep=0.75pt]  [font=\scriptsize] [align=left] {x3};
% Text Node
\draw (51,251) node [anchor=north west][inner sep=0.75pt]   [align=left] {(i)};
% Text Node
\draw (221,232) node [anchor=north west][inner sep=0.75pt]   [align=left] {(ii)};
% Text Node
\draw (324.67,406) node [anchor=north west][inner sep=0.75pt]   [align=left] {TV A};
% Text Node
\draw (515.33,470.67) node [anchor=north west][inner sep=0.75pt]   [align=left] {TV B};
% Text Node
\draw (527.33,382) node [anchor=north west][inner sep=0.75pt]   [align=left] {TV C};
% Text Node
\draw (331,591) node [anchor=north west][inner sep=0.75pt]   [align=left] {Regulation at TV C, then at TV A };
% Text Node
\draw (321,551) node [anchor=north west][inner sep=0.75pt]  [font=\scriptsize] [align=left] {x3};
% Text Node
\draw (316,492) node [anchor=north west][inner sep=0.75pt]  [font=\scriptsize] [align=left] {x1};

\end{tikzpicture}
    
    \caption{The ordering of regulation affects the flows available for the next step. ``xN'' means the flow contains N flights.}
    \label{fig:regulation_order_example}
\end{figure}
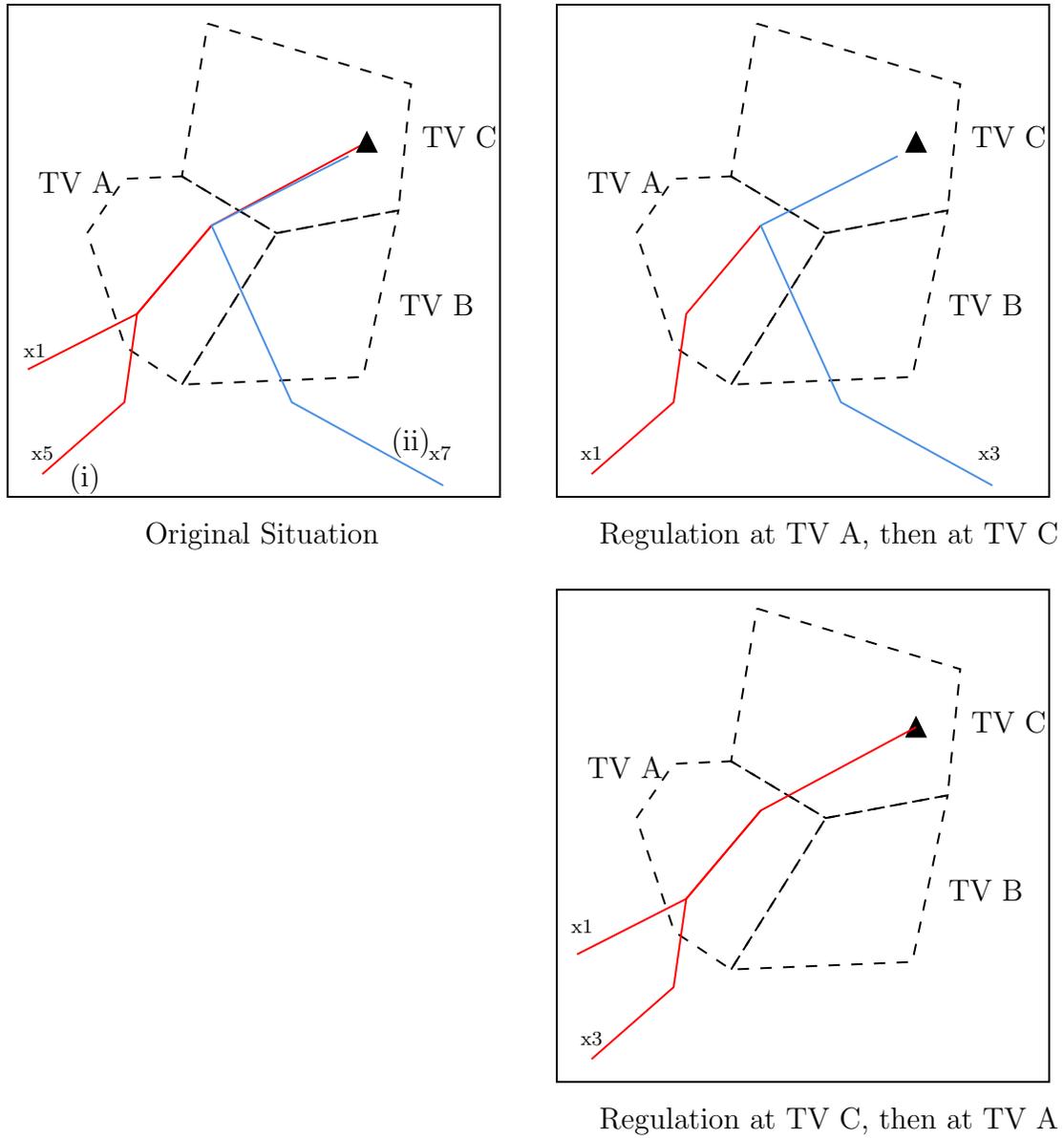

\subsection{Problem Formulation}
Owing to the structural resemblance of our problem to strategic decision-making in \textit{chess}, we adopt the \textit{AlphaZero} framework \cite{silver2018general} to derive the optimal \textit{sequence of regulations} that minimizes DCB imbalances while accounting for delay propagation. Similar to AlphaZero, our framework leverages MCTS to balance exploration and exploitation under the principle of optimism \cite{keller2012prost}. At each iteration, the algorithm simulates the application of candidate regulations and propagates the observed rewards backward through the search tree, progressively refining its policy. As the exploration bonus decays over time, the algorithm converges toward a greedy strategy that prioritizes the most frequently visited nodes, corresponding to actions with the highest estimated value.

\subsubsection{Definitions}
Fix a bin width of $\Delta = 15$ minutes. Let the planning day be partitioned into $96$ consecutive time bins indexed by
\[
\mathcal{T} \equiv \{0,1,\ldots,95\}.
\]
For $t \in \mathcal{T}$, bin $t$ corresponds to the half-open interval $[t\Delta,(t+1)\Delta)$ of minutes after midnight.
For integers $a \le b$, we use the discrete interval notation
\[
[a,b]_{\mathbb{Z}} \equiv \{a,a+1,\ldots,b\}.
\]

Let $\mathcal{V}$ denote the finite set of traffic volumes (TFVs). A generic element is $v \in \mathcal{V}$. Let $\mathcal{F}$ denote the set of flights considered in the planning horizon. For a flight $f \in \mathcal{F}$ that traverses volume $v$, let $e_{v,f} \in [0,24\mathrm{h})$ denote the (continuous) entry time into $v$ and $x_{v,f} \in (0,24\mathrm{h}]$ the exit time, with $e_{v,f} < x_{v,f}$; if $f$ does not traverse $v$, these times are undefined.
Define the entry-bin map $b(t) \equiv \lfloor t/\Delta \rfloor$ so that the entry bin of $f$ into $v$ is $b(e_{v,f}) \in \mathcal{T}$.

For each $v \in \mathcal{V}$ and $t \in \mathcal{T}$, define the within-bin entry count
\[
E_{v,t} \;\equiv\; \bigl|\{f \in \mathcal{F} : b(e_{v,f}) = t \}\bigr|.
\]
The rolling-hour (or a 60-min sliding window) entry demand starting at bin $t$ is
\[
D_{v,t} \;\equiv\; \sum_{k=0}^{3} E_{v,\,t + k},
\]
i.e., the number of entries into $v$ over the hour $[t\Delta, (t+4)\Delta)$.

Let $C_{v,t} \in \mathbb{N}$ denote the allowable hourly entry capacity for volume $v$ associated with the rolling window starting at bin $t$. We define the excess (overload) as
\[
G_{v,t} \;\equiv\; D_{v,t} - C_{v,t}.
\]

A hotspot segment for volume $v$ is a maximal contiguous discrete interval $[t_i,t_f]_{\mathbb{Z}} \subseteq \mathcal{T}$ such that
\[
G_{v,t} > 0, \quad \forall\, t \in [t_i,t_f]_{\mathbb{Z}}.
\]
We denote such a segment by $h \equiv (v, t_i, t_f)$. Maximality means that $t_i=0$ or $G_{v,t_i-1}\le 0$, and $t_f=95$ or $G_{v,t_f+1}\le 0$.

We also define the instantaneous occupancy
\[
n_v(u) \;\equiv\; \bigl|\{ f \in \mathcal{F} : e_{v,f} \le u < x_{v,f} \}\bigr|, \qquad u \in [0,24\mathrm{h}).
\]
Its bin-aggregated average is
\[
\bar{n}_{v,t} \;\equiv\; \frac{1}{\Delta} \int_{t\Delta}^{(t+1)\Delta} n_v(u)\,du,
\]
and the rolling-hour average occupancy is $\sum_{k=0}^{3} \bar{n}_{v,\,t + k}$. Note that rate regulations act on entries ($D_{v,t}$), not on occupancy.

A regulation is a quadruple
\[
\rho \;=\; \bigl(v_c,\, t_{i,c},\, t_{f,c},\, \tau_c \bigr),
\]
where $v_c \in \mathcal{V}$ is the control volume, $[t_{i,c}, t_{f,c}]_{\mathbb{Z}} \subseteq \mathcal{T}$ is the regulation’s nominal active period, and $\tau_c \in \mathbb{N}$ is the hourly entry-rate limit (entries per hour) imposed at $v_c$.

Because $D_{v,t}$ is defined on sliding-hour windows, effective regulation may begin up to three bins (45 minutes) earlier to affect the first constrained window and may extend beyond the nominal end to avoid post-regulation count bunching. We therefore define regulation margins
\[
\Delta_{e}, \Delta_{r} \in \mathbb{N}_0 \quad\text{with}\quad 0 \le \Delta_{e} \le 3,\ \ \Delta_{r} \ge 1,
\]
and the effective active window
\[
I_c^{\mathrm{eff}} \;\equiv\; [\,t_{i,c} - \Delta_e,\; t_{f,c} + \Delta_r\,]_{\mathbb{Z}} \cap \mathcal{T}
\]

During $I_c^{\mathrm{eff}}$, the regulated entry demand satisfies
\[
D_{v_c,t}^{(\rho)} \;\le\; \tau_c, \qquad \forall\, t \in I_c^{\mathrm{eff}},
\]
where $D_{v_c,t}^{(\rho)}$ denotes the rolling-hour entry demand at $v_c$ after the regulation-induced metering/queuing is applied.

A (partial) regulation plan is an ordered finite sequence
\[
s \;=\; (\rho^1, \rho^2, \ldots, \rho^K), \qquad K \in \mathbb{N}_0,
\]
with $\rho^j$ regulations as above. Its cardinality is $|s| = K$. We view $s$ as the state of the decision process. Given a plan $s$, we write $D_{v,t}^{(s)}$ for the induced rolling-hour entry demand and $C_{v,t}^{(s)}$ for the effective capacity profile after applying all regulations in $s$.
\subsubsection{FPFS Slot Allocation (FIFO Queue)}\label{subsec:fpfs_slot}
For a regulation $\rho=(v_c,t_{i,c},t_{f,c},\tau_c)$ with margins $\Delta_e,\Delta_r$, let the effective continuous-time window be
\[
W_c^{\mathrm{eff}} \equiv \bigl[\,a_c,\, b_c\,\bigr), 
\qquad a_c \equiv (t_{i,c}-\Delta_e)\Delta,\quad b_c \equiv (t_{f,c}+\Delta_r+1)\Delta,
\]
where $\Delta=15$ minutes.
Let the slot spacing be
\[
\sigma_c \equiv \frac{60}{\tau_c}\ \text{minutes}, \qquad \tau_c \in \mathbb{N},\ \tau_c \ge 1,
\]
and define the slot sequence anchored at $a_c$ by
\[
s_{c,0} \equiv a_c,\qquad s_{c,m} \equiv s_{c,0} + m\,\sigma_c,\quad m \in \mathbb{N}_0.
\]
Let the set of regulated flights be
\[
\mathcal{F}_c \equiv \{\, f \in \mathcal{F} : \text{$f$ traverses $v_c$ and } e_{v_c,f} \in W_c^{\mathrm{eff}} \,\}.
\]
Order $\mathcal{F}_c$ by nondecreasing planned entry times and a fixed deterministic tie-breaker (e.g., flight identifier): 
\[
e_{v_c,f^{(1)}} \le e_{v_c,f^{(2)}} \le \cdots \le e_{v_c,f^{(N_c)}},\quad N_c \equiv |\mathcal{F}_c|.
\]
Define $m_0 \equiv -1$ and, for $j=1,\ldots,N_c$,
\begin{equation}
m_j \equiv \max\!\left\{\, m_{j-1}+1,\ \left\lceil \frac{e_{v_c,f^{(j)}} - s_{c,0}}{\sigma_c} \right\rceil \right\}, 
\qquad 
\hat e_{v_c,f^{(j)}}^{(\rho)} \equiv s_{c,0} + m_j\,\sigma_c.
\end{equation}
The concrete delay (in minutes) assigned to $f^{(j)}$ by $\rho$ is
\begin{equation}
d_{f^{(j)}}^{(\rho)} \equiv \hat e_{v_c,f^{(j)}}^{(\rho)} - e_{v_c,f^{(j)}} \;\;\in\; [0,\infty).
\end{equation}
For flights not subject to $\rho$ (i.e., $f \notin \mathcal{F}_c$), set $\hat e_{v_c,f}^{(\rho)} \equiv e_{v_c,f}$ and $d_f^{(\rho)} \equiv 0$.

\subsubsection{Objective Function}

Given a (partial) regulation plan $s=(\rho^1,\ldots,\rho^K)$, we define plan-dependent quantities and a scalar objective $J(s)$ to be minimized.

Let $\hat e_{v,f}^{(s)}$ denote the realized (post-regulation) entry time of flight $f \in \mathcal{F}$ into volume $v \in \mathcal{V}$ after applying all regulations in $s$; in the absence of regulation, $\hat e_{v,f}^{(s)} \equiv e_{v,f}$. Define the bin map $b(u) \equiv \lfloor u/\Delta \rfloor$ with $\Delta=15$ minutes and the induced within-bin entry counts
\[
E_{v,t}^{(s)} \;\equiv\; \bigl|\{\, f \in \mathcal{F} : b(\hat e_{v,f}^{(s)}) = t \,\}\bigr|,\qquad v \in \mathcal{V},\ t \in \mathcal{T}.
\]
The induced rolling-hour entry demand is
\[
D_{v,t}^{(s)} \;\equiv\; \sum_{k=0}^{3} E_{v,\,t + k}^{(s)},\qquad v \in \mathcal{V},\ t \in \mathcal{T}.
\]
For a flight $f$, let $\mathcal{V}(f) \subseteq \mathcal{V}$ be the set of volumes traversed by $f$. The net delay (in minutes) assigned to $f$ by plan $s$ is
\[
d_f^{(s)} \;\equiv\; \max_{v \in \mathcal{V}(f)} \bigl( \hat e_{v,f}^{(s)} - e_{v,f} \bigr)_+,
\]
where $(x)_+ \equiv \max\{x,0\}$ denotes the positive part. This definition avoids double-counting when multiple regulations constrain the same flight.

Let $C_{v,t} \in \mathbb{N}$ be the hourly entry capacity for volume $v$ associated with the rolling window starting at bin $t$. For weights $w_{\mathrm{DELAY}}, w_{\mathrm{REG}}, w_{\mathrm{TV}} \ge 0$, define
\begin{equation}
J(s) \;=\; J_{\mathrm{CAP}}(s) \;+\; w_{\mathrm{DELAY}}\, J_{\mathrm{DELAY}}(s) \;+\; w_{\mathrm{REG}}\, J_{\mathrm{REG}}(s) \;+\; w_{\mathrm{TV}}\, J_{\mathrm{TV}}(s),
\end{equation}
with components:
\[
J_{\mathrm{CAP}}(s) \;\equiv\; \sum_{v \in \mathcal{V}} \sum_{t \in \mathcal{T}} \bigl(D_{v,t}^{(s)} - C_{v,t}\bigr)_+,
\]
i.e., the total post-regulation excess demand over all volumes and time bins;
\[
J_{\mathrm{DELAY}}(s) \;\equiv\; \sum_{f \in \mathcal{F}} d_f^{(s)},
\]
the total delay (minutes) assigned to all flights;
\[
J_{\mathrm{REG}}(s) \;\equiv\; |s| \;=\; K,
\]
the number of regulations issued; and
\[
J_{\mathrm{TV}}(s) \;\equiv\; \sum_{v \in \mathcal{V}} \sum_{t=1}^{95} \bigl| E_{v,t}^{(s)} - E_{v,t-1}^{(s)} \bigr|,
\]
the total variation of within-bin entry counts, which penalizes bunching and sharp, short-lived peaks in the regulated flow.

\paragraph{Optimization problem statement.}
Let $\mathcal{S}$ be the set of admissible regulation plans $s=(\rho^1,\ldots,\rho^K)$ with $K \in \mathbb{N}_0$ and $\rho^j=(v_c^j,t_{i,c}^j,t_{f,c}^j,\tau_c^j)$ such that $v_c^j \in \mathcal{V}$, $t_{i,c}^j \le t_{f,c}^j \in \mathcal{T}$, and $\tau_c^j \in \mathbb{N}$. Given fixed margins $(\Delta_e,\Delta_r)$ and the FPFS allocator, the plan-induced quantities $E^{(s)}, D^{(s)}, d^{(s)}$ are as defined above. The planning problem is to solve for $s^*$ that minimizes $J(s)$:
\[
\boxed{s^* = \arg\min_{\,s \in \mathcal{S}} \; J(s)}.
\]

\subsubsection{Architecture Overview}

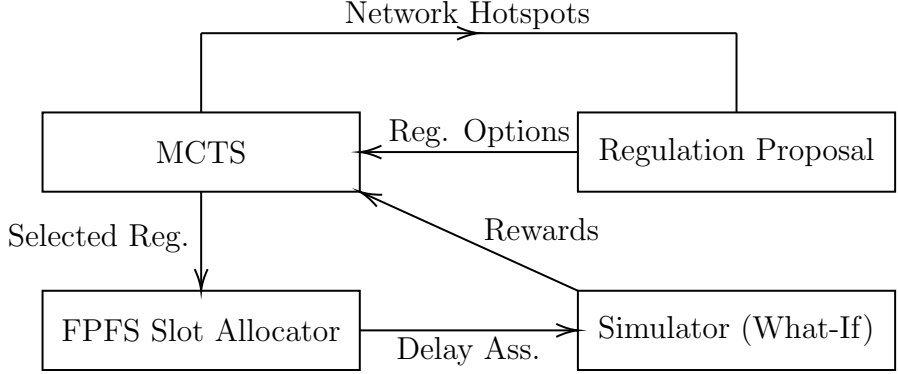
\begin{figure}
\centering\begin{tikzpicture}[x=0.75pt,y=0.75pt,yscale=-1,xscale=1]

%Shape: Rectangle [id:dp7529089092279516] 
\draw   (340,70) -- (500,70) -- (500,110) -- (340,110) -- cycle ;
%Shape: Rectangle [id:dp2215918374688578] 
\draw   (70,70) -- (230,70) -- (230,110) -- (70,110) -- cycle ;
%Shape: Rectangle [id:dp2890862689026613] 
\draw   (340,160) -- (500,160) -- (500,200) -- (340,200) -- cycle ;
%Shape: Rectangle [id:dp6545542204656801] 
\draw   (70,160) -- (230,160) -- (230,200) -- (70,200) -- cycle ;
%Straight Lines [id:da5611531222562754] 
\draw    (340,90) -- (232,90) ;
\draw [shift={(230,90)}, rotate = 360] [color={rgb, 255:red, 0; green, 0; blue, 0 }  ][line width=0.75]    (10.93,-3.29) .. controls (6.95,-1.4) and (3.31,-0.3) .. (0,0) .. controls (3.31,0.3) and (6.95,1.4) .. (10.93,3.29)   ;
%Straight Lines [id:da2525975745678619] 
\draw    (150,110) -- (150,158) ;
\draw [shift={(150,160)}, rotate = 270] [color={rgb, 255:red, 0; green, 0; blue, 0 }  ][line width=0.75]    (10.93,-3.29) .. controls (6.95,-1.4) and (3.31,-0.3) .. (0,0) .. controls (3.31,0.3) and (6.95,1.4) .. (10.93,3.29)   ;
%Straight Lines [id:da7579396188519203] 
\draw    (230,180) -- (338,180) ;
\draw [shift={(340,180)}, rotate = 180] [color={rgb, 255:red, 0; green, 0; blue, 0 }  ][line width=0.75]    (10.93,-3.29) .. controls (6.95,-1.4) and (3.31,-0.3) .. (0,0) .. controls (3.31,0.3) and (6.95,1.4) .. (10.93,3.29)   ;
%Straight Lines [id:da4273769422599779] 
\draw    (340,160) -- (231.82,110.83) ;
\draw [shift={(230,110)}, rotate = 24.44] [color={rgb, 255:red, 0; green, 0; blue, 0 }  ][line width=0.75]    (10.93,-3.29) .. controls (6.95,-1.4) and (3.31,-0.3) .. (0,0) .. controls (3.31,0.3) and (6.95,1.4) .. (10.93,3.29)   ;
%Straight Lines [id:da33985762968936817] 
\draw    (150,30) -- (288,30) ;
\draw [shift={(290,30)}, rotate = 180] [color={rgb, 255:red, 0; green, 0; blue, 0 }  ][line width=0.75]    (10.93,-3.29) .. controls (6.95,-1.4) and (3.31,-0.3) .. (0,0) .. controls (3.31,0.3) and (6.95,1.4) .. (10.93,3.29)   ;
%Straight Lines [id:da6754707823993692] 
\draw    (150,30) -- (150,70) ;
%Straight Lines [id:da5157119079684226] 
\draw    (420,30) -- (420,70) ;
%Straight Lines [id:da8720565810420458] 
\draw    (290,30) -- (420,30) ;

% Text Node
\draw (420,90) node   [align=left] {\begin{minipage}[lt]{108.8pt}\setlength\topsep{0pt}
\begin{center}
Regulation Proposal
\end{center}

\end{minipage}};
% Text Node
\draw (150,90) node   [align=left] {\begin{minipage}[lt]{108.8pt}\setlength\topsep{0pt}
\begin{center}
MCTS
\end{center}

\end{minipage}};
% Text Node
\draw (420,180) node   [align=left] {\begin{minipage}[lt]{108.8pt}\setlength\topsep{0pt}
\begin{center}
Simulator (What-If)
\end{center}

\end{minipage}};
% Text Node
\draw (150,180) node  [font=\normalsize] [align=left] {\begin{minipage}[lt]{108.8pt}\setlength\topsep{0pt}
\begin{center}
FPFS Slot Allocator
\end{center}

\end{minipage}};
% Text Node
\draw (243,72) node [anchor=north west][inner sep=0.75pt]   [align=left] {Reg. Options};
% Text Node
\draw (51,125) node [anchor=north west][inner sep=0.75pt]   [align=left] {Selected Reg.};
% Text Node
\draw (247,182) node [anchor=north west][inner sep=0.75pt]   [align=left] {Delay Ass.};
% Text Node
\draw (291,122) node [anchor=north west][inner sep=0.75pt]   [align=left] {Rewards};
% Text Node
\draw (221,12) node [anchor=north west][inner sep=0.75pt]   [align=left] {Network Hotspots};

\end{tikzpicture}
    \caption{RegulationZero Key Components}
    \label{fig:rz_key_components}
\end{figure}

Figure~\ref{fig:rz_key_components} outlines the key components of the RegulationZero framework. The MCTS maintains a search tree representing partial regulation plans and their associated rewards. In each simulation, the search begins from a given partial plan $s_0$ and expands by selecting a hotspot where a new regulation might be applied. The Regulation Proposal engine then takes the hotspot’s flight list, groups related flights through community detection, and evaluates a set of candidate rate reductions which involves an FPFS Slot Allocator (Section \ref{subsec:fpfs_slot}). The most promising proposals will be selected and added to the tree. The trajectory's total reward is the sum of total objective improvements accumulated so far.

Below, we will discuss the components in more details.
\subsection{Regulation Proposal Engine}\label{subsec:proposal_engine}
Given a hotspot segment $h \equiv (v_h, t_{i,h}, t_{f,h})$ selected by the MCTS from a current partial plan $s$, the proposal engine enumerates a small, high-quality set of candidate regulations $\rho=(v_c,t_{i,c},t_{f,c},\tau_c)$ with $v_c=v_h$ to expand the search tree. Each candidate is evaluated by the FPFS allocator (Section \ref{subsec:fpfs_slot}) on the full traffic set to ensure feasibility with respect to sliding-hour counts, and is scored by its incremental improvement in the objective $J$ (Section 2.3). The engine proceeds in three steps: (i) extraction of flows relevant to $h$, (ii) community detection over a flight-similarity graph derived from path “footprints,” and (iii) generation and scoring of rate-window proposals informed by community structure.

\subsubsection{Hotspot-relevant flight set}
Fix a look-back margin $\Delta_h \in \{0,1,2,3\}$ to capture the 60-min rolling-window coupling at the hotspot. We define the hotspot influence interval
$$
I_h \;\equiv\; [\,t_{i,h}-\Delta_h,\ t_{f,h}\,]_{\mathbb{Z}} \cap \mathcal{T},
$$
and the associated flight set
$$
\mathcal{F}_h \;\equiv\; \bigl\{ f \in \mathcal{F} : \text{$f$ traverses $v_h$ and } b(e_{v_h,f}) \in I_h \bigr\}.
$$
By construction, $\mathcal{F}_h$ contains all flights whose planned entries into $v_h$ contribute to at least one overloaded rolling-hour bin in $[t_{i,h},t_{f,h}]_{\mathbb{Z}}$, together with those that may be advanced into these windows by queueing dynamics.

\subsubsection{Flight footprints}
For a flight $f \in \mathcal{F}$, let its route through traffic volumes be the finite sequence
$$
L_f \;\equiv\; \bigl( v_1(f), v_2(f), \ldots, v_{K_f}(f) \bigr), \qquad v_k(f) \in \mathcal{V},
$$
ordered by increasing traversal time. We call $L_f$ the \textit{TV footprint of the flight} $f$. This allows us to define a distance metric that quantifies the \textit{impact similarity} of flights contributing to the hotspot:

Given $f,g \in \mathcal{F}_h$, we define their demand-impact distance as the Jaccard distance between their footprints:
\begin{equation}\label{eq:jaccard}
d_h(f,g) \;\equiv\; 1 \;-\; \frac{\bigl| L_{f} \cap L_{g} \bigr|}{\bigl| L_{f} \cup L_{g} \bigr|} \;\;\in\; [0,1],
\end{equation}
with the convention $d_h(f,g)=0$ if both sets are empty. Intuitively, $d_h$ is small for flights that traverse largely overlapping volumes near the hotspot and large otherwise, thus serving as a proxy for the potential of a single regulation to coherently meter both flights.

Figure~\ref{fig:jaccardexample} illustrates the computation of the Jaccard distance. For any two flights contributing to the hotspot EGLMU during a specific time segment, their corresponding TV footprints are retrieved, and the number of shared TVs is determined. The Jaccard distance is then obtained by dividing the cardinality of the intersection by that of the union of the two footprints, yielding a value of $1/3$.

\begin{figure}
	\centering
	\begin{tikzpicture}[x=0.75pt,y=0.75pt,yscale=-1,xscale=1]
	
	%Shape: Rectangle [id:dp29000641997001564] 
	\draw   (210,60) -- (280,60) -- (280,80) -- (210,80) -- cycle ;
	%Shape: Rectangle [id:dp20890898894283794] 
	\draw   (290,60) -- (360,60) -- (360,80) -- (290,80) -- cycle ;
	%Shape: Rectangle [id:dp15242621676145895] 
	\draw   (370,60) -- (440,60) -- (440,80) -- (370,80) -- cycle ;
	%Shape: Rectangle [id:dp8939361399848287] 
	\draw   (450,60) -- (520,60) -- (520,80) -- (450,80) -- cycle ;
	%Shape: Rectangle [id:dp18046825038709635] 
	\draw   (210,90) -- (280,90) -- (280,110) -- (210,110) -- cycle ;
	%Shape: Rectangle [id:dp48849583619770853] 
	\draw   (290,90) -- (360,90) -- (360,110) -- (290,110) -- cycle ;
	%Shape: Rectangle [id:dp2562376347593255] 
	\draw   (370,90) -- (440,90) -- (440,110) -- (370,110) -- cycle ;
	%Shape: Rectangle [id:dp3477681156101696] 
	\draw   (450,90) -- (520,90) -- (520,110) -- (450,110) -- cycle ;
	
	% Text Node
	\draw (245,70) node  [font=\scriptsize] [align=left] {\begin{minipage}[lt]{47.6pt}\setlength\topsep{0pt}
			\begin{center}
				LFMWMFDZ
			\end{center}
			
	\end{minipage}};
	% Text Node
	\draw (325,70) node  [font=\scriptsize,color={rgb, 255:red, 255; green, 0; blue, 0 }  ,opacity=1 ] [align=left] {\begin{minipage}[lt]{47.6pt}\setlength\topsep{0pt}
			\begin{center}
				EGLMU
			\end{center}
			
	\end{minipage}};
	% Text Node
	\draw (405,70) node  [font=\scriptsize] [align=left] {\begin{minipage}[lt]{47.6pt}\setlength\topsep{0pt}
			\begin{center}
				LFMAIET
			\end{center}
			
	\end{minipage}};
	% Text Node
	\draw (485,70) node  [font=\scriptsize] [align=left] {\begin{minipage}[lt]{47.6pt}\setlength\topsep{0pt}
			\begin{center}
				LFMMFDZ
			\end{center}
			
	\end{minipage}};
	% Text Node
	\draw (245,100) node  [font=\scriptsize] [align=left] {\begin{minipage}[lt]{47.6pt}\setlength\topsep{0pt}
			\begin{center}
				LFMWMFDZ
			\end{center}
			
	\end{minipage}};
	% Text Node
	\draw (325,100) node  [font=\scriptsize,color={rgb, 255:red, 255; green, 0; blue, 0 }  ,opacity=1 ] [align=left] {\begin{minipage}[lt]{47.6pt}\setlength\topsep{0pt}
			\begin{center}
				EGLMU
			\end{center}
			
	\end{minipage}};
	% Text Node
	\draw (405,100) node  [font=\scriptsize] [align=left] {\begin{minipage}[lt]{47.6pt}\setlength\topsep{0pt}
			\begin{center}
				LFMRAES
			\end{center}
			
	\end{minipage}};
	% Text Node
	\draw (485,100) node  [font=\scriptsize] [align=left] {\begin{minipage}[lt]{47.6pt}\setlength\topsep{0pt}
			\begin{center}
				LFFOPKZ
			\end{center}
			
	\end{minipage}};
	% Text Node
	\draw (365,150) node   [align=left] {\begin{minipage}[lt]{346.8pt}\setlength\topsep{0pt}
			$\displaystyle {\textstyle d_{h}( f,g) =\frac{|\{LFMWMFDZ,\ EGLMU\} |}{|\{LFMWMFDZ,\ EGLMU,LFMAIET,LFMMFDZ,LFMRAES,LFFOPKZ\}} =\frac{1}{3}}$
	\end{minipage}};
	% Text Node
	\draw (176,62) node [anchor=north west][inner sep=0.75pt]   [align=left] {$\displaystyle f$};
	% Text Node
	\draw (176,90) node [anchor=north west][inner sep=0.75pt]   [align=left] {$\displaystyle g$};

\end{tikzpicture}
	\caption{Demand-impact distance calculation for two flights $f, g$ given the hotspot $h$ EGLMU.}
	\label{fig:jaccardexample}
\end{figure}

\subsubsection{Graph construction and community detection for flow extraction}\label{subsub:flow_x}
We construct an undirected, weighted similarity graph
$$
G_h \;=\; (\mathcal{V}_h^{\mathrm{graph}},\, \mathcal{E}_h,\, w), \qquad \mathcal{V}_h^{\mathrm{graph}} \equiv \mathcal{F}_h,
$$
with edge weights
$$
w_{fg} \;=\; \begin{cases}
1 \quad\text{if $d_h(f,g) > \tau_{min}$,}\\
0 \quad\text{otherwise.}
\end{cases}
$$

We then partition $\mathcal{F}_h$ into communities $\mathcal{P}_h=\{\CC_1,\ldots,\CC_R\}$ using the Leiden algorithm, which iteratively refines an aggregation to maximize a weighted modularity \cite{traag2019louvain}. 

Each triple $(\CC,[t_i,t_f],\tau)$ induces a potential regulation candidate
$$
\rho(\CC,[t_i,t_f],\tau) \;\equiv\; (\, v_h,\ t_i,\ t_f,\ \tau \,).
$$

\begin{figure}
	\centering
	\includegraphics[width=.6\linewidth]{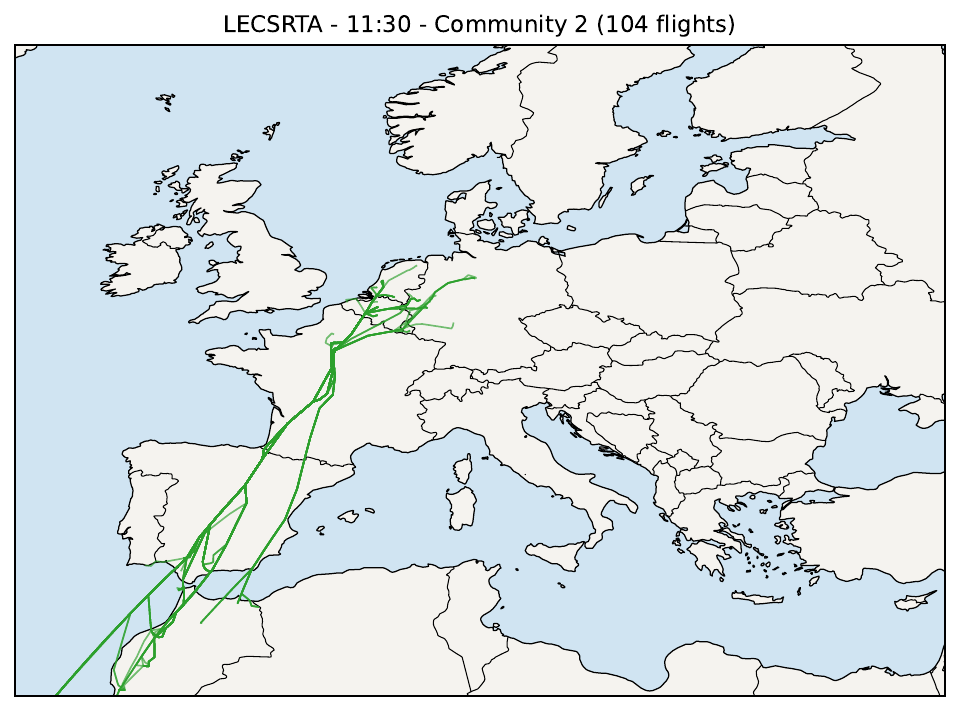}
	\caption{A sample flow extracted from the Community Detection algorithm that is relevant to the hotspot segment LECSRTA (11:30-12:15, 23/07/2023) that could be subjected to regulation. The flow contains flights using Madrid and Paris airspaces.} \label{fig:flow_x_sample}
\end{figure}

\subsubsection{Flow Selection Heuristics}
The flow extraction step in Section~\ref{subsub:flow_x} yields a collection of candidate flows 
\[
\mathcal{P}_h = \{\, \CC_1,\, \CC_2,\,\ldots,\, \CC_R \,\}
\]
associated with the hotspot $h \equiv (v_h, t_{i,h}, t_{f,h})$. Because individual flows often traverse multiple TFVs concurrently, we introduce three heuristic metrics to measure a flow's potential to alleviate simultaneous overloads across several volumes, while mitigating the risk of inducing secondary imbalances elsewhere in the network.

\vspace{4pt}
\paragraph{Flow Pressure.}
To quantify the overall contribution of a flow $\CC_j$ to network-wide overloads, we define its \emph{flow pressure} as the cumulative excess demand induced by the flights in $\CC_j$ across all affected TFVs and rolling-hour bins.
Let
\[
\mathcal{V}(\CC_j) \;\equiv\; \bigcup_{f \in \CC_j} \mathcal{V}(f)
\]
denote the set of volumes traversed by at least one flight in $\CC_j$.
For each $v \in \mathcal{V}(\CC_j)$ and $t \in \mathcal{T}$, let $\mathcal{F}_{v,t}(\CC_j) \subseteq \CC_j$ be the subset of flights of $\CC_j$ entering $v$ during bin $t$, and define their induced demand in that bin as
\[
E_{v,t}(\CC_j) \;\equiv\; |\mathcal{F}_{v,t}(\CC_j)|.
\]
The corresponding rolling-hour contribution of $\CC_j$ at $(v,t)$ is
\[
D_{v,t}(\CC_j) \;\equiv\; \sum_{k=0}^{3} E_{v,\,t+k}(\CC_j).
\]

Then, the flow pressure for $\CC_j$ during the current plan $s$ is
\begin{equation}\label{eq:flow_pressure}
	P_h(C_j\,|\,s)
	\;\equiv\;
	\sum_{v \in \mathcal{V}(\CC_j)} 
	\sum_{t \in [t_{i,h}-\Delta_h,\,t_{f,h}]_{\mathbb{Z}} \cap \mathcal{T}}
	\Bigl[ \,
	D_{v,t}^{(s)} 
	- C_{v,t}
	\,\Bigr]_{
		+}
	\;\cdot\;
	\frac{
		D_{v,t}(\CC_j)
	}{
		\displaystyle |\mathcal{F}_{v,t}|
	}.
\end{equation}
The first component, $\Bigl[\,D_{v,t}^{(s)} - C_{v,t}\,\Bigr]_{+}$, represents the total overload observed at bin $(v,t)$. The second component quantifies the share of this overload that can be attributed to the flow $\CC_j$. The normalization factor ensures that the measure is not biased toward larger traffic flows merely due to their higher number of flights. Conceptually, $P_h(\CC_j\,|\,s)$ captures the degree to which the flights comprising $\CC_j$ are involved in simultaneous capacity exceedances throughout the network.

\begin{figure}
	\centering
	\begin{tikzpicture}[x=0.75pt,y=0.75pt,yscale=-1,xscale=1]
	
	%Shape: Polygon [id:ds5259506919111093] 
	\draw  [fill={rgb, 255:red, 241; green, 241; blue, 241 }  ,fill opacity=1 ] (260,30) -- (310,60) -- (260,140) -- (250,90) -- (210,60) -- cycle ;
	%Shape: Polygon [id:ds4106769573864326] 
	\draw  [fill={rgb, 255:red, 241; green, 241; blue, 241 }  ,fill opacity=1 ] (210,60) -- (250,90) -- (260,140) -- (180,160) -- (150,100) -- cycle ;
	%Shape: Polygon [id:ds04792757400979675] 
	\draw   (310,60) -- (340,120) -- (350,180) -- (270,190) -- (260,140) -- cycle ;
	%Shape: Polygon [id:ds10906584828725552] 
	\draw   (260,140) -- (270,190) -- (240,210) -- (200,210) -- (180,160) -- cycle ;
	%Curve Lines [id:da5005981112178421] 
	\draw [color={rgb, 255:red, 0; green, 73; blue, 255 }  ,draw opacity=1 ]   (120,160) .. controls (245.12,135.52) and (235.12,98.72) .. (320,30) ;
	%Curve Lines [id:da29721989909694957] 
	\draw [color={rgb, 255:red, 0; green, 73; blue, 255 }  ,draw opacity=1 ]   (160,210) .. controls (332.72,184.16) and (235.12,98.72) .. (320,30) ;
	
	% Text Node
	\draw (161,96) node [anchor=north west][inner sep=0.75pt]  [font=\footnotesize] [align=left] {TV A};
	% Text Node
	\draw (251,46) node [anchor=north west][inner sep=0.75pt]  [font=\footnotesize] [align=left] {TV B};
	% Text Node
	\draw (201,192) node [anchor=north west][inner sep=0.75pt]  [font=\footnotesize] [align=left] {TV C};
	% Text Node
	\draw (311,162) node [anchor=north west][inner sep=0.75pt]  [font=\footnotesize] [align=left] {TV D};
	% Text Node
	\draw (92,141) node [anchor=north west][inner sep=0.75pt]   [align=left] {Flow X};
	% Text Node
	\draw (121,191) node [anchor=north west][inner sep=0.75pt]   [align=left] {Flow Y};

\end{tikzpicture}
	\caption{Flow Heuristic Example Setting.}
	\label{fig:flowheuristics}
\end{figure}
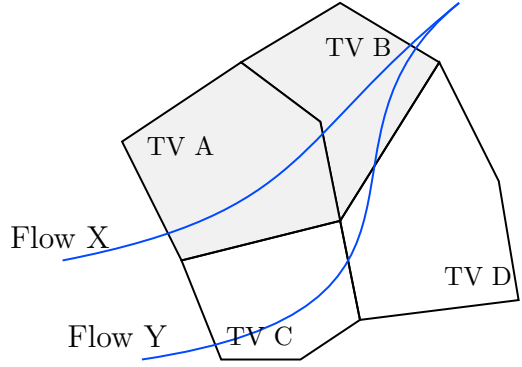

In the example scenario depicted in Figure \ref{fig:flowheuristics}, Flow X has higher pressure since it has the potential to relieve two overloaded hotspots (TV A and B), compared to only one (Flow Y).

\paragraph{Flow's Slacks.} We now introduce a family of metrics designed to quantify the remaining capacity headroom along a flow's footprint. These metrics capture the spare capacity of the most constrained TFV encountered by the flow. A small slack value implies that the local capacity margin is narrow, suggesting that interventions may heighten the risk of secondary hotspots emerging.

For $p,q \in \mathcal{V}$, define the empirical, directional travel-time set
\[
\mathcal{T}_{p \to q}
\;\equiv\;
\bigl\{\, e_{q,f} - e_{p,f} \,:\, f \in \mathcal{F},\ \text{$f$ traverses $p$ then $q$ with } e_{p,f} < e_{q,f} \,\bigr\}.
\]
The reference inter-TV travel-time is the robust median
\[
T(p \to q) \;\equiv\; \operatorname{median}\bigl(\mathcal{T}_{p \to q}\bigr).
\]
which could be approximated as following should computation is difficult:
\[
T(p \to q) \;\equiv\; \frac{\mathrm{dist}(\mathrm{centroid}(p),\,\mathrm{centroid}(q))}{V_{\mathrm{ref}}},
\]
where $V_{\mathrm{ref}}$ is a fixed reference ground speed and $\mathrm{dist}(\cdot,\cdot)$ is the great-circle distance. For a hotspot $h \equiv (v_h,t_{i,h},t_{f,h})$ and any $v \in \mathcal{V}$, we write the signed offset
\[
\Delta_{v,h} \;\equiv\; T(v \to v_h),
\]
with the convention $\Delta_{v,h} > 0$ if $v$ is (on average) upstream of $v_h$, and $\Delta_{v,h} < 0$ if $v$ is downstream.

Fix a hotspot $h=(v_h,t_{i,h},t_{f,h})$, a current partial plan $s$, and a candidate flow $\CC \subseteq \mathcal{F}$ extracted for $h$ (Section~\ref{subsub:flow_x}). Let $\Delta=15$ minutes and $b(u)\equiv\lfloor u/\Delta\rfloor$ be the bin map. For a nominal additional delay $M \in \mathbb{R}_{+}$ (minutes) induced by a slot allocator (e.g., CASA) at $v_h$, the flights in $\CC$ that would have entered the hotspot around bin $t$ are expected to load another TV $v$ around the aligned bin
\[
\phi(v,t;M) \;\equiv\; b\!\bigl(t\Delta + M - \Delta_{v,h}\bigr) \;\in\; \mathcal{T}.
\]
Define the flow’s $M$-slack as the worst-case residual capacity, across its footprint:
\begin{equation}
	\label{eq:flow_slack}
	\mathrm{Slack}_M(\CC \mid h,s)
	\;\equiv\;
	\min_{\substack{v \in \mathcal{V}(\CC)\\ t \in I_h}}
	\Bigl\{
	C_{v,\,\phi(v,t;M)}
	\;-\;
	D_{v,\,\phi(v,t;M)}^{(s)}
	\;+\;
	D_{v,\,\phi(v,t;M)}(\CC)
	\Bigr\}.
\end{equation}
Here:
\begin{itemize}
	\item $\mathcal{V}(\CC)=\bigcup_{f\in\CC}\mathcal{V}(f)$ is the set of TVs traversed by at least one flight in $\CC$;
	\item $I_h=[t_{i,h}-\Delta_h,\,t_{f,h}]_{\mathbb{Z}} \cap \mathcal{T}$ is the hotspot influence interval (Section~\ref{subsec:proposal_engine});
	\item $D_{v,t}^{(s)}$ is the current rolling-hour demand at $(v,t)$ under plan $s$;
	\item $D_{v,t}(\CC)$ is the rolling-hour contribution of $\CC$ at $(v,t)$ as in Section~\ref{subsec:proposal_engine}.
\end{itemize}
The formulation can be interpreted as computing the residual capacity after removing the flow's own contribution from the total demand. In other words, it measures how much spare capacity remains if the flow were not participating in the regulation. By taking the minimum over all $v \in \mathcal{V}(\CC)$ and $t \in I_h$, \eqref{eq:flow_slack} identifies the most restrictive point along the flow’s trajectory, or the tightest bottleneck it would encounter when a regulation is applied at $v_h$ with the nominal delay $M$.

In our framework, the regulation of a flow $\CC$ does not change its TV footprint $\mathcal{V}(\CC)$; it primarily shifts the entry times of its flights. Consequently, any secondary hotspot that a regulation might induce must occur within $\mathcal{V}(\CC)$. We define two measures:
\[
\mathrm{\texttt{Slack15}}(\CC \mid h,s) \;\equiv\; \mathrm{Slack}_{M=15}(\CC \mid h,s), 
\qquad
\mathrm{\texttt{Slack30}}(\CC \mid h,s) \;\equiv\; \mathrm{Slack}_{M=30}(\CC \mid h,s).
\]
\begin{itemize}
	\item \texttt{Slack15} reflects how easily the flow can be modestly metered (small rate reduction) without creating new overloads elsewhere. A positive value indicates sufficient headroom across the footprint for approximately 15 minutes of additional delay at the hotspot.
	\item \texttt{Slack30} reflects the robustness of more aggressive metering (larger rate reduction).
\end{itemize}

Finally, we define a surrogate function for flow selection: 
\begin{equation}\label{eq:flow_selection_heuristic}
\Phi(\CC_j) = w_{pressure} P_h (\CC_j | s) + w_{Slack15} Slack15(\CC_j) + w_{Slack30} Slack30(\CC_j)
\end{equation}

\subsection{Regulations Proposing}
Given a hotspot $h=(v_h,t_{i,h},t_{f,h})$ selected from the current partial plan $s$, we perform a targeted local search that converts the flow structure around $h$ into a small set of high-value regulation candidates. The search returns up to $K$ proposals $\rho=(v_h,t_{i,h},t_{f,h},\tau)$, each evaluated with the FPFS slot allocator (Section~\ref{subsec:fpfs_slot}) on the full traffic set and scored by its incremental improvement $\Delta J \equiv J(s)-J(s\cup\rho)$.

\paragraph{Inputs and Hyperparameters.}
\begin{itemize}
	\item Flow set $\mathcal{P}_h = \{\CC_1, \ldots, \CC_R\}$ - the set of candidate flows extracted for hotspot $h$.
	\item max\_flows\_in\_regulation $(M_{\max})$:  maximum number of flows allowed to be jointly regulated (default: $5$).
	\item top-$k$ proposals $(k_{\mathrm{top}})$: number of best-performing regulation proposals returned to the MCTS.
	\item min\_flights\_per\_flow $(N_{\min})$: minimum number of flights a flow must contain to be considered eligible.
	\item Grid of rate factors $\Gamma$: discrete set of candidate rate multipliers used for grid search, e.g.\ $\Gamma = \{0.6, 0.7, 0.8, 0.9, 1.0\}$.
\end{itemize}

\paragraph{Rate initialization.}
Let $C_{\min}(h)\equiv \min_{t\in I_h^{\mathrm{eff}}} C_{v_h,t}$ be a conservative per-hour capacity bound for the effective window. For a set of flows $S\subseteq \mathcal{P}_h$, define the demand share within $I_h^{\mathrm{eff}}$ as
\[
p(S\,|\,s)\;\equiv\; 
\frac{\sum_{t\in I_h^{\mathrm{eff}}} \omega_t\, D_{v_h,t}(S)}
{\sum_{t\in I_h^{\mathrm{eff}}} \omega_t\, D_{v_h,t}^{(s)}},
\qquad 
\omega_t\;\equiv\; \bigl(D_{v_h,t}^{(s)}-C_{v_h,t}\bigr)_+ + \varepsilon,
\]
with a small $\varepsilon > 0$ added to avoid division by zero. The weights $\omega_t$ emphasize bins experiencing positive overloads. The nominal (or baseline) rate is defined as 
\[
\tau_0(S) \;\equiv\; \mathrm{round}\bigl( p(S\,|\,s) \bigr),
\]
and the set of candidate rates explored during grid search is given by multiplicative perturbations of this nominal value:
\[
\mathcal{T}(S) \;=\; 
\bigl\{\, \tau_0(S) \cdot (1+\gamma) \;:\; \gamma \in \Gamma \cup (-\Gamma) \,\bigr\}.
\]

The search explores unions of the top-$r$ flows (by $\phi$) for $r=1,\ldots,R_{\max}$ and, for each union, evaluates a set of rates around $\tau_0$. Each candidate regulation is simulated via FPFS on the full traffic set, and scored by $\Delta J$.

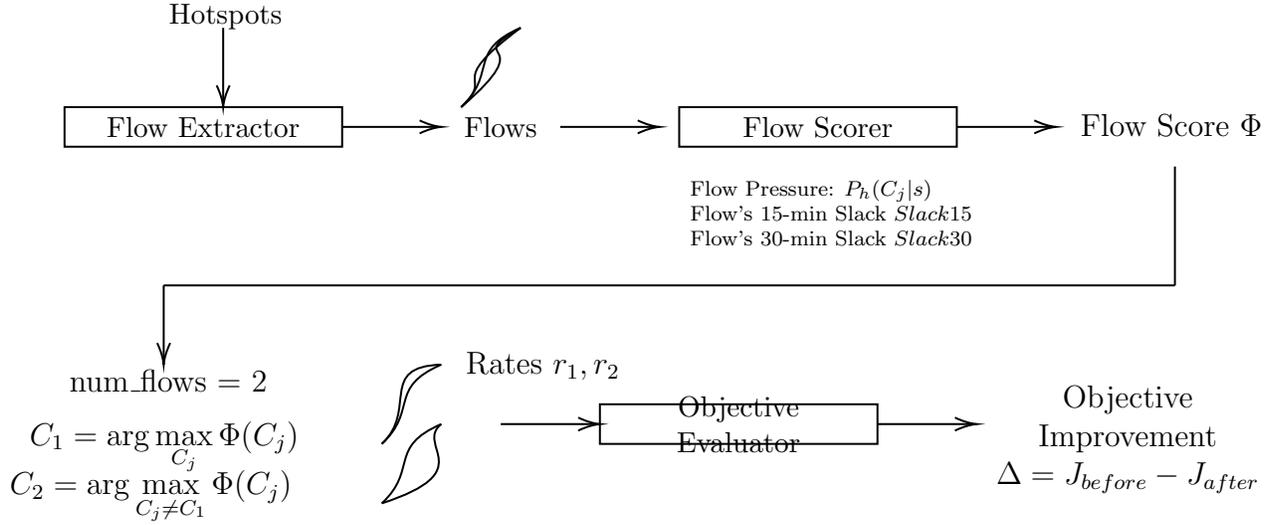
\begin{figure}
	\centering
	\begin{tikzpicture}[x=0.75pt,y=0.75pt,yscale=-1,xscale=1]
	%uncomment if require: \path (0,352); %set diagram left start at 0, and has height of 352
	
	%Shape: Rectangle [id:dp20235290235352732] 
	\draw   (40,110) -- (180,110) -- (180,130) -- (40,130) -- cycle ;
	%Straight Lines [id:da8844572195915404] 
	\draw    (180,120) -- (228,120) ;
	\draw [shift={(230,120)}, rotate = 180] [color={rgb, 255:red, 0; green, 0; blue, 0 }  ][line width=0.75]    (10.93,-3.29) .. controls (6.95,-1.4) and (3.31,-0.3) .. (0,0) .. controls (3.31,0.3) and (6.95,1.4) .. (10.93,3.29)   ;
	%Straight Lines [id:da02965738911852256] 
	\draw    (120,70) -- (120,108) ;
	\draw [shift={(120,110)}, rotate = 270] [color={rgb, 255:red, 0; green, 0; blue, 0 }  ][line width=0.75]    (10.93,-3.29) .. controls (6.95,-1.4) and (3.31,-0.3) .. (0,0) .. controls (3.31,0.3) and (6.95,1.4) .. (10.93,3.29)   ;
	%Shape: Rectangle [id:dp5700553494158938] 
	\draw   (350,110) -- (490,110) -- (490,130) -- (350,130) -- cycle ;
	%Straight Lines [id:da8650755265656416] 
	\draw    (290,120) -- (338,120) ;
	\draw [shift={(340,120)}, rotate = 180] [color={rgb, 255:red, 0; green, 0; blue, 0 }  ][line width=0.75]    (10.93,-3.29) .. controls (6.95,-1.4) and (3.31,-0.3) .. (0,0) .. controls (3.31,0.3) and (6.95,1.4) .. (10.93,3.29)   ;
	%Straight Lines [id:da004324629031796534] 
	\draw    (490,120) -- (538,120) ;
	\draw [shift={(540,120)}, rotate = 180] [color={rgb, 255:red, 0; green, 0; blue, 0 }  ][line width=0.75]    (10.93,-3.29) .. controls (6.95,-1.4) and (3.31,-0.3) .. (0,0) .. controls (3.31,0.3) and (6.95,1.4) .. (10.93,3.29)   ;
	%Straight Lines [id:da6307165353555734] 
	\draw    (260,270) -- (308,270) ;
	\draw [shift={(310,270)}, rotate = 180] [color={rgb, 255:red, 0; green, 0; blue, 0 }  ][line width=0.75]    (10.93,-3.29) .. controls (6.95,-1.4) and (3.31,-0.3) .. (0,0) .. controls (3.31,0.3) and (6.95,1.4) .. (10.93,3.29)   ;
	%Shape: Rectangle [id:dp6219188606232103] 
	\draw   (310,260) -- (450,260) -- (450,280) -- (310,280) -- cycle ;
	%Straight Lines [id:da4443301826635476] 
	\draw    (450,270) -- (498,270) ;
	\draw [shift={(500,270)}, rotate = 180] [color={rgb, 255:red, 0; green, 0; blue, 0 }  ][line width=0.75]    (10.93,-3.29) .. controls (6.95,-1.4) and (3.31,-0.3) .. (0,0) .. controls (3.31,0.3) and (6.95,1.4) .. (10.93,3.29)   ;
	%Straight Lines [id:da7726942418187308] 
	\draw    (600,140) -- (600,200) ;
	%Straight Lines [id:da3387865045998921] 
	\draw    (90,200) -- (600,200) ;
	%Straight Lines [id:da9655657465812642] 
	\draw    (90,200) -- (90,238) ;
	\draw [shift={(90,240)}, rotate = 270] [color={rgb, 255:red, 0; green, 0; blue, 0 }  ][line width=0.75]    (10.93,-3.29) .. controls (6.95,-1.4) and (3.31,-0.3) .. (0,0) .. controls (3.31,0.3) and (6.95,1.4) .. (10.93,3.29)   ;
	%Curve Lines [id:da2683578980821897] 
	\draw    (240,110) .. controls (258.33,93) and (230,100) .. (270,70) ;
	%Curve Lines [id:da1092336869734486] 
	\draw    (240,110) .. controls (258.33,93) and (243,81.67) .. (270,70) ;
	%Curve Lines [id:da5110672386327739] 
	\draw    (240,110) .. controls (281.67,74.33) and (239,87.67) .. (270,70) ;
	%Curve Lines [id:da06778261268916497] 
	\draw    (200,280) .. controls (218.33,263) and (203,251.67) .. (230,240) ;
	%Curve Lines [id:da03236574467720399] 
	\draw    (200,310) .. controls (250.33,293.67) and (216.33,281) .. (230,270) ;
	%Curve Lines [id:da8140114546427386] 
	\draw    (200,280) .. controls (218.33,263) and (197.67,241) .. (230,240) ;
	%Curve Lines [id:da3307217225871302] 
	\draw    (200,310) .. controls (218.33,293) and (203,281.67) .. (230,270) ;
	
	% Text Node
	\draw (110,120) node  [font=\small] [align=left] {\begin{minipage}[lt]{95.2pt}\setlength\topsep{0pt}
			\begin{center}
				Flow Extractor
			\end{center}
			
	\end{minipage}};
	% Text Node
	\draw (260,120) node  [font=\small] [align=left] {\begin{minipage}[lt]{40.8pt}\setlength\topsep{0pt}
			\begin{center}
				Flows
			\end{center}
			
	\end{minipage}};
	% Text Node
	\draw (120,60) node  [font=\small] [align=left] {\begin{minipage}[lt]{40.8pt}\setlength\topsep{0pt}
			\begin{center}
				Hotspots
			\end{center}
			
	\end{minipage}};
	% Text Node
	\draw (420,120) node  [font=\small] [align=left] {\begin{minipage}[lt]{95.2pt}\setlength\topsep{0pt}
			\begin{center}
				Flow Scorer
			\end{center}
			
	\end{minipage}};
	% Text Node
	\draw (354,145) node [anchor=north west][inner sep=0.75pt]  [font=\scriptsize] [align=left] {Flow Pressure: $\displaystyle P_{h}( C_{j} |s)$\\Flow's 15-min Slack $\displaystyle Slack15$\\Flow's 30-min Slack $\displaystyle Slack30$};
	% Text Node
	\draw (551,112) node [anchor=north west][inner sep=0.75pt]   [align=left] {Flow Score $\displaystyle \Phi $};
	% Text Node
	\draw (41,241) node [anchor=north west][inner sep=0.75pt]   [align=left] {num\_flows = 2};
	% Text Node
	\draw (21,267) node [anchor=north west][inner sep=0.75pt]   [align=left] {$\displaystyle C_{1} =\arg\max_{C_{j}} \Phi ( C_{j})$};
	% Text Node
	\draw (11,292) node [anchor=north west][inner sep=0.75pt]   [align=left] {$\displaystyle C_{2} =\arg\max_{C_{j} \neq C_{1}} \Phi ( C_{j})$};
	% Text Node
	\draw (380,270) node  [font=\small] [align=left] {\begin{minipage}[lt]{95.2pt}\setlength\topsep{0pt}
			\begin{center}
				Objective Evaluator
			\end{center}
			
	\end{minipage}};
	% Text Node
	\draw (503,249) node [anchor=north west][inner sep=0.75pt]   [align=left] {\begin{minipage}[lt]{107.6pt}\setlength\topsep{0pt}
			\begin{center}
				Objective Improvement\\$\displaystyle \Delta =J_{before} -J_{after}$
			\end{center}
			
	\end{minipage}};
	% Text Node
	\draw (241,232) node [anchor=north west][inner sep=0.75pt]   [align=left] {Rates $\displaystyle r_{1} ,r_{2}$};

\end{tikzpicture}
	\caption{The Regulation Proposal Engine's Main Components, with \texttt{num\_flows = 2}.}
	\label{fig:regen}
\end{figure}

Figure \ref{fig:regen} provides a high-level view of the entire data-flows within the Regulation Proposal engine employed by Regulation Zero.
%

% ================================
% SECTION: RegulationZero via MCTS
% Requires: \usepackage{amsmath,amssymb,mathtools,algorithm,algpseudocode}
% ================================

\subsection{RegulationZero: Hierarchical Monte-Carlo Tree Search for Regulation Sequence Optimization}
\label{subsec:regzero-mcts}
\subsubsection{Flow Regulations as a Markov Decision Process Problem}
As discussed earlier, it is essential to solve the traffic regulations in sequence. This motivates us to formulate the regulation search as planning in a finite-horizon Markov decision process (MDP)
\(
\mathcal{M} = (\mathcal{S}, \mathcal{A}, P, R, \gamma)
\),
with state space \(\mathcal{S}\), \emph{joint} action space
\(\mathcal{A} = \mathcal{R}\times\mathcal{H}\) where \(a=(r,h)\) combines a regulation proposal \(r\in\mathcal{R}\) and a hotspot \(h\in\mathcal{H}\),
transition kernel \(P(\cdot \mid s,a)\),
reward \(R(s,a) \doteq \Delta J(s,a)\) given by the objective improvement
(weighted reduction of exceedances and delay),
and discount \(\gamma \in (0,1]\).
For an episode \(s_0,s_1,\dots,s_T\), the objective is to maximize the expected discounted return
\begin{equation}
	J^\pi(s_0) \;=\; \mathbb{E}_\pi\!\left[ \sum_{t=0}^{T-1} \gamma^t \,\Delta J_t \right],
	\quad \text{with } s_{t+1}\sim P(\cdot \mid s_t,a_t), \; a_t\sim \pi(\cdot\mid s_t).
	\label{eq:objective}
\end{equation}

Because applying regulations results in a deterministic outcome:
\begin{equation*}
	P(s_{t+1} | s_t, a_t) = \delta_{\{s = s_{t+1}\}}.
\end{equation*}

\begin{figure}
	\centering
	\input{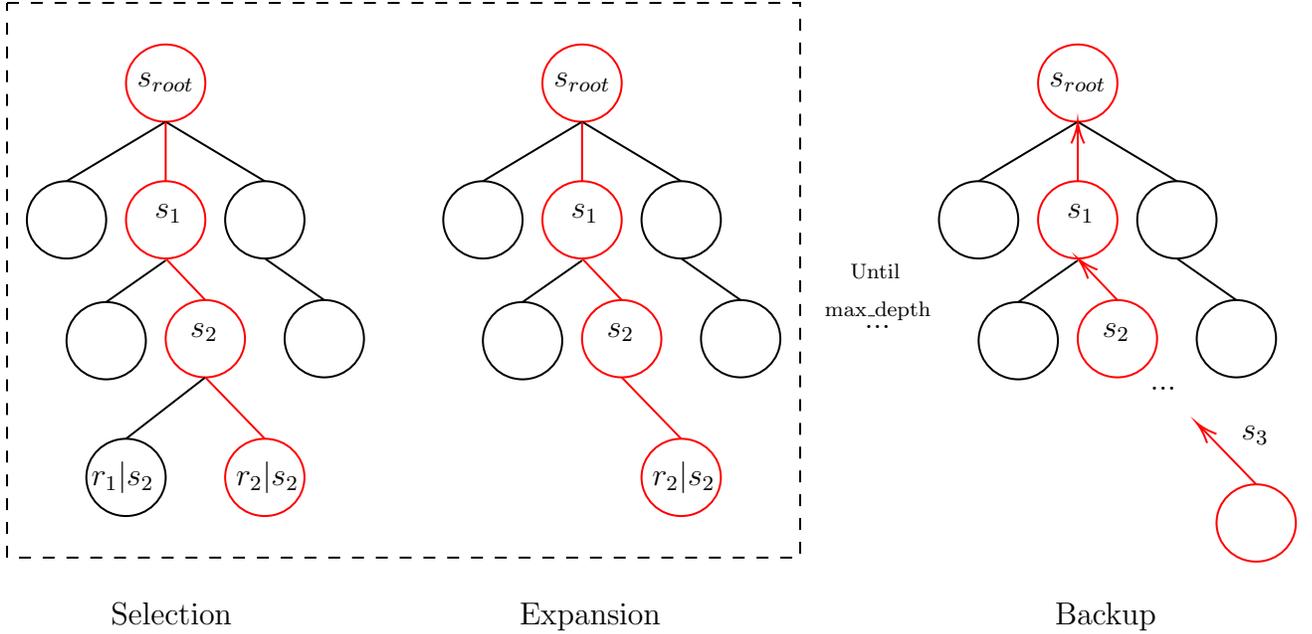}
	\caption{Illustration of the three main phases of the MCTS algorithm employed by RegulationZero.}
	\label{fig:mcts}
\end{figure}

The key steps of our MCTS implementation version can be described in Figure \ref{fig:mcts}. 

\subsubsection{Hierarchical action and priors.}
Because the flow extraction result $\mathcal{P}_h$ is tied to the selected hotspot segment $h$, we factor the search prior policy as
\[
\pi(a\mid s) \;=\; \pi_H(h\mid s)\,\pi_R(r\mid h,s),
\]
and assume that the priors $P_H(h\mid s)>0$ and $P_R(r\mid h,s)>0$. 

Let $\mathcal{H}(s)$ be the set of hotspot segments. We define the severity of a hotspot as the cumulative exceedance across its segment under plan $s$:
\begin{equation}\label{eq:phi_h_def}
	\phi(h \mid s)
	\;\equiv\;
	\sum_{t \in [t_{i,h},\,t_{f,h}]_{\mathbb{Z}}}
	\Bigl( D_{v_h,t}^{(s)} - C_{v_h,t} \Bigr)_{+},
\end{equation}
where $(x)_{+} \equiv \max\{x,0\}$.

We sample hotspots according to a tempered softmax:
\begin{equation}\label{eq:PH}
	P_H(h \mid s)
	\;=\;
	\frac{\exp\!\bigl(\phi(h\mid s)/\tau_H\bigr)}
	{\displaystyle \sum_{h'\in \mathcal{H}(s)} \exp\!\bigl(\phi(h'\mid s)/\tau_H\bigr)}
\end{equation}

Conditioned on $h$, the \texttt{RegulationProposal} engine returns a finite candidate set
$\mathcal{R}(h,s)=\{r_1,\ldots,r_m\}$ together with predicted immediate improvements $\widehat{\Delta J}(s,h,r)$. Then, the regulation to be selected will follow the PUCT rule.

\subsubsection{PUCT-driven proposal selection}
Within the search tree, we select proposals using a PUCT criterion instantiated for the \((h,r)\) branching factorization:
\begin{equation}
	r^* \;=\; \arg\max_{r}
	\left[
	Q(s,h,r) \;+\; c_{\mathrm{PUCT}}\, P_R(r\mid h,s)\,
	\frac{\sqrt{\sum_b n(s,h,b)}}{1+n(s,h,r)}
	\right],
	\label{eq:puct_selection}
\end{equation}
where \(Q(s,h,r)\) is the current value estimate and \(n(s,h,r)\) counts visits to child \((s,h,r)\).

The prior \(P_R(r\mid h,s)\) guides the search process to explore new regulations in the beginning of the search phase, but gradually subsides, rendering the policy becoming more greedy as total visits $\sum_b n(s,h,b)$ increase.

After a regulation is selected and applied, the corresponding delays are imposed on the affected flights, and the flight list is updated accordingly. This update triggers a re-evaluation of active hotspots and a new flow extraction step. The reward obtained at iteration $k$ is defined as $\Delta J_k = \hat{\Delta J}(s, h, r^*)$. The algorithm then continues to extend the sequence until reaching a predefined search depth.

\subsubsection{Value backups}
At the end of the trajectory, the cumulative discounted reward is:
\begin{equation}
G \!=\! \sum_{k} \gamma^{k-1}\Delta J_k.
\end{equation}

For each visited node on the path, we perform the following updates:
\begin{align}
	n(s,h,r) &\leftarrow n(s,h,r) + 1,\\
	W(s,h,r) &\leftarrow W(s,h,r) + G,\\
	Q(s,h,r) &\leftarrow \frac{W(s,h,r)}{n(s,h,r)}.
\end{align}

All the steps of the MCTS can be summarized in Algorithm \ref{alg:regzero}.

\begin{algorithm}[t]
	\caption{RegulationZero: Hierarchical, KL-Regularized MCTS}
	\label{alg:regzero}
	\begin{algorithmic}[1]
		\State \textbf{Input:} priors \(P_H, P_R\); severity \(\phi\); discount \(\gamma\); \(c_{\mathrm{PUCT}}\); \texttt{max\_depth}
		\For{simulation $=1,\dots,N$}
		\State $s \leftarrow s_0$; path $\leftarrow [\;]$; $d \leftarrow 0$
		\While{$d < \texttt{max\_depth}$ and $s$ not terminal}
		\State \textbf{Hotspot:} sample $h \sim \pi_H(\cdot\mid s) \propto P_H(\cdot\mid s)\exp(\phi/\tau_H)$
		\If{some $(s,h, \cdot)$ unexpanded} \State retrieve/update $\{r(h)\}$ via \texttt{RegulationProposal} \EndIf
		\State \textbf{Proposal:} select $r$ by \eqref{eq:puct_selection}
		\State Append $(s,h,r)$ to path
		\State Execute $a=(h,r)$; observe $\Delta J$; sample $s' \sim P(\cdot\mid s,a)$
		\State $s \leftarrow s'$; $d \leftarrow d+1$
		\EndWhile
		\State Compute return $G \leftarrow \sum_{t=0}^{d-1} \gamma^t \Delta J_t$
		\For{each node $(s,h,r)$ in path} \State update $n,W,Q$ with $G$ \EndFor
		\EndFor
		\State \textbf{return} action at root with largest visit count or value
	\end{algorithmic}
\end{algorithm}

\subsection{Comments on Architecture Design}
In this section, we provide an overarching rationale for the design of our framework. Conventional approaches in reinforcement learning (RL) studies often formulate the action space around atomic operations, such as adding or removing a regulation. However, our experiments indicate that this design frequently leads to cyclic exploration patterns, as the vast combinatorial space of possible regulatory modifications causes the policy to become trapped in repetitive loops. While one could impose a simulation budget to curb unbounded exploration, the long trajectory lengths inherent to such settings substantially weaken the gradient signal during backpropagation, thereby impeding effective learning.

The role of the regulation proposal (RP) is fundamental in ensuring the stability of the MCTS search process. One valid view of RP is that RP is the local-search extension of MCTS's exploration at each node. Not only would this help narrowing down the search space, but 

Following the analysis presented in Appendix~\ref{sec:poliopt}, \textit{RegulationZero} optimizes a KL-regularized objective that promotes structured exploration at both the hotspot and regulation levels:
\begin{equation}
	\Omega_s(\pi) = \tau_h \, \mathrm{KL}\!\left[\pi_H(\cdot\,|\,s) \,\|\, P_H(\cdot\,|\,s)\right] 
	+ \mathbb{E}_{h \sim \pi_H}\!\left[\lambda_R (s,h) \, \mathrm{KL}\!\left[P_R(\cdot\,|\,h,s) \,\|\, \pi_R(\cdot\,|\,h,s)\right]\right].
\end{equation}

The first term, $\tau_h \, \mathrm{KL}[\pi_H(\cdot|s) \| P_H(\cdot|s)]$, represents a reverse KL-divergence akin to that used in the Soft Actor-Critic (SAC) formulation \cite{haarnoja2018soft}, favoring mode-seeking behavior that concentrates attention on the most critical hotspots $h$.

In contrast, the second term, $\mathrm{KL}[P_R(\cdot|h,s) \| \pi_R(\cdot|h,s)]$, takes the form of a forward KL-divergence, encouraging mode-covering exploration within the regulation space.

Conceptually, this separation guides the policy to focus relief efforts on the most severe hotspots, while simultaneously enabling a comprehensive and diverse search across the regulatory landscape. The empirical results discussed in the following section illustrate how these dual objectives manifest in practice, notably through the temporary steps that deliberately ``worsens the situation'' before improving it, despite better options were available at disposal.

\section{Results}
To validate the proposed framework, we conducted a comprehensive set of evaluations using planned flight plan data obtained from EUROCONTROL’s DDR2 exported from NEST. However, we reimplemented the Simulator (What-If engine) instead of using R-NEST for CASA on our open ATFM research platform Cortex \footnote{Source code is available at \url{https://github.com/thinhhoang95/project-cortex}.}. Performance was benchmarked against a Simulated Annealing (SA) baseline to assess both solution quality and computational efficiency. All experiments were carried out on a dedicated cloud instance equipped with an AMD Ryzen Threadripper 7960X processor and 64 GB of RAM. All implementations are realized in Python 3.13 with Numba. All objective values are reported with weight $w_{CAP} = 10, w_{DELAY} = 1$ or one occupancy exceedance is equivalent to 10 minutes of delay.
\subsection{Performance \& Results at Scales}
For this section, we focus on comparing the performance and optimality of solutions of RegulationZero compared against the SA baseline. We employ the initial traffic dataset consisting of 23,089 flights from midnight to end-of-day of 23rd July, 2023.
\begin{table}[htbp]
	\centering % Center the table
	\caption{Hyperparameter Configuration for RegulationZero}
	\label{tab:hyperparams}
	% This is the table from the previous snippet
	\begin{tabular}{ll}
		\toprule
		\textbf{Parameter} & \textbf{Value} \\
		\midrule
		sims & 128 \\
		depth & 64 \\
		commit\_depth & 64 \\
		flows\_threshold & 0.6 \\
		flows\_resolution & 1.0 \\
		max\_hotspots\_per\_node & 12 \\
		k\_proposals\_per\_hotspot & 5 \\
		puct\_c & 64 \\
		gamma & 0.999998 \\
		regulation\_selection\_softmax\_temperature & 24.0 \\
		hotspot\_sampling\_temperature & 6.0 \\
		\bottomrule
	\end{tabular}
\end{table}

\begin{table}[htbp]
	\centering
	% Define the table: 0.5\linewidth ensures it's centered and not too wide.
	% The column types are:
	% X: Automatically calculated width (used for descriptive Parameter Name)
	% c: Center-aligned (used for the Value)
	
	\caption{Hyperparameters used for the optimization algorithm.}
	\label{tab:hyperparameterssa}
	\begin{tabularx}{0.5\linewidth}{X c}
		\toprule
		\textbf{Parameter Name} & \textbf{Value} \\
		\midrule
		{iters} & $10,000$ \\
		{T0} (Initial Temperature) & $15.0$ \\
		{cooling} (Cooling Rate) & $0.999$ \\
		{T\_min} (Minimum Temperature) & $1 \times 10^{-9}$ \\
		{max\_delay\_per\_flight\_min} & $240$ \\
		\bottomrule
	\end{tabularx}
\end{table}
The hyperparameters of RegulationZero are set according to Table \ref{tab:hyperparams}, the SA hyperparameters are given in Table \ref{tab:hyperparameterssa}. These hyperparameters were refined from the automatically tuned values retrieved from Optuna with 15 run trials to give the best results for both algorithms.

We divide the flight set into three run cases: the three-hour case that spans from 10:00 to 13:00 where most hotspots occurred, the six-hour case that spans from 10:00 to 16:00, and the whole day of operation case. The goal is to test the effectiveness of RegulationZero against the problem size. A sample of RZ's output can be seen from Figure \ref{fig:regulation_trajectory}.

\begin{figure}[htb!]
	\centering
	% Subfigure 1
	\begin{subfigure}[b]{0.99\linewidth} % Adjust width to fit 3 side-by-side (e.g., 0.3\linewidth for 3 figures with some space)
		\centering
		\includegraphics[width=\linewidth]{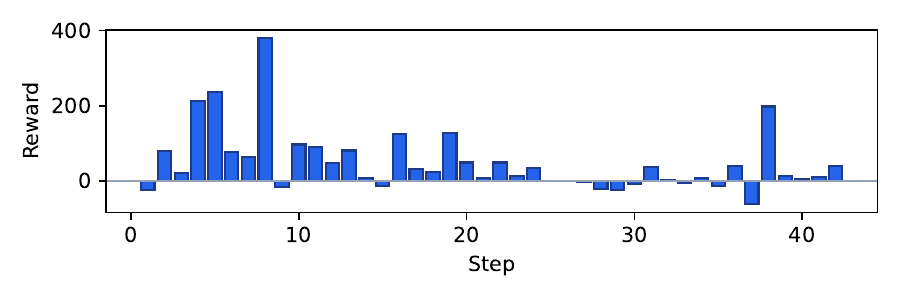}
		\caption{3h case}
		\label{fig:3hstagerewards}
	\end{subfigure}
	% Subfigure 2
	\begin{subfigure}[b]{0.99\linewidth}
		\centering
		\includegraphics[width=\linewidth]{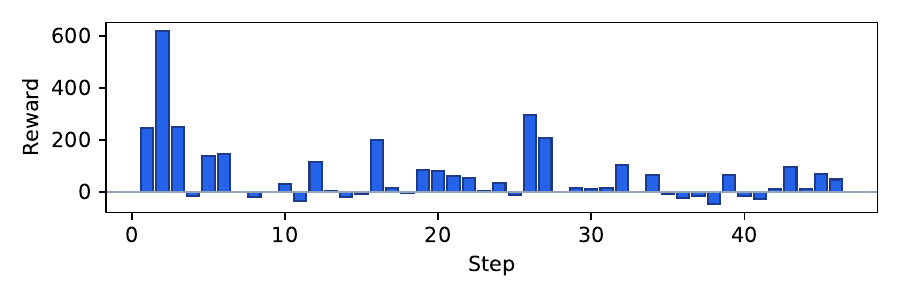}
		\caption{6h case}
		\label{fig:6hstagerewards}
	\end{subfigure}=
	% Subfigure 3
	\begin{subfigure}[b]{0.99\linewidth}
		\centering
		\includegraphics[width=\linewidth]{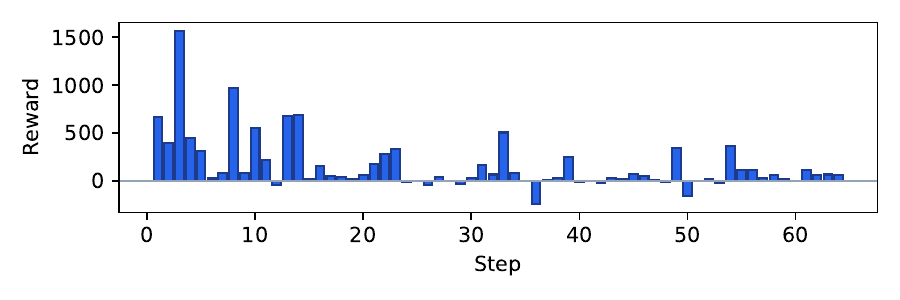}
		\caption{Whole-day case}
		\label{fig:24hstagerewards}
	\end{subfigure}
	
	\caption{Per-step reward obtained by RegulationZero for the optimal regulation sequence for the 3h, 6h, and whole-day cases, respectively.}
	\label{fig:combinedstagerewards} % Main figure label
\end{figure}

When the commit depth is set to 64, we observe that RegulationZero tends to exhaust its available regulatory commit budget in an effort to further optimize the objective function. In the 3h case, which contains the largest number of hotspots, it is challenging for the regulatory mechanism to fully dissipate congestion without merely shifting it elsewhere. RegulationZero often deploys numerous small-scale regulations that yield limited or even negative direct benefit, yet help reorient traffic patterns so that subsequent regulatory actions can exert a more substantial effect (Figure \ref{fig:3hstagerewards}). This pattern is particularly evident for regulations applied later in the sequence. Similar patterns can be seen in the 6h and the whole-day case (Figure \ref{fig:6hstagerewards} and \ref{fig:24hstagerewards}), though the late ``impactful'' regulations did not occur as these cases were generally easier to solve without lots of traffic reorientation.

The mode-seeking behavior of hotspots is clearly reflected in these reward trajectories: the most impactful regulations are typically executed first and target the most severe hotspots. Consequently, RegulationZero, like many other search and optimization algorithms, exhibits anytime performance, meaning that the search can be terminated early when a less optimal but acceptable solution suffices.

\begin{table}[htbp]
	\centering % Center the table on the page
	\caption{Comparison of RZ Performance on Three Problem Sizes}
	\label{tab:rz_meta}
	% Start tabularx, making it the full text width.
	% l: standard left-aligned column for the criteria
	% >{\centering\arraybackslash}X: three "X" columns which will expand equally
	%                                to fill the width. The >{...} part
	%                                centers the text within these columns.
	\begin{tabularx}{\textwidth}{l >{\centering\arraybackslash}X >{\centering\arraybackslash}X >{\centering\arraybackslash}X}
		\toprule % Top rule from booktabs
		% Header row - using \textbf for emphasis
		\textbf{Criteria} & \textbf{3h} & \textbf{6h} & \textbf{Whole-day} \\
		\midrule % Middle rule from booktabs
		% Data rows
		Number of Regulations & 68 & 67 & 100 \\
		Number of Flights Involved & 1193 & 1156 & 2145 \\
		%Number of Flights Delayed & 243 & 266 & 525 \\
		Max Run Time (s) & 12742 & 10144 & 12801 \\
		\bottomrule % Bottom rule from booktabs
	\end{tabularx}
\end{table}

\begin{table}[htbp]
	\centering % Center the table on the page
	\caption{Comparison of SA Performance on Three Problem Sizes}
	\label{tab:sa_meta}
	\begin{tabularx}{\textwidth}{l >{\centering\arraybackslash}X >{\centering\arraybackslash}X >{\centering\arraybackslash}X}
		\toprule % Top rule
		
		% Header row
		\textbf{Criteria} & \textbf{3h} & \textbf{6h} & \textbf{Whole Day} \\
		
		\midrule % Middle rule
		
		% Data rows
		%Number of Flights Delayed & 308 & 312 & 344 \\
		Acceptance Rate & 0.482 & 0.489 & 0.485 \\
		Max Run Time (s) & 6419 & 6552 & 6448 \\		
		\bottomrule % Bottom rule
	\end{tabularx}
\end{table}

\begin{table}[htbp]
	\centering % Center the table on the page
	\caption{Comparison of Results for RZ and SA Methods}
	\label{tab:rzsa}
	% Define 5 columns: l (left-aligned) for Case, and four X (expanding, centered) columns.
	\begin{tabularx}{\textwidth}{l *{6}{>{\centering\arraybackslash}X}}
		\toprule % Top rule
		% Header Row 1: Merged Columns
		\textbf{Case} & 
		\multicolumn{2}{c}{\textbf{Final Objective Improvements}} & 
		\multicolumn{2}{c}{\textbf{Total Delay (min)}} &
		\multicolumn{2}{c}{\textbf{Number of Flights}} \\
		
		\cmidrule(lr){2-3} \cmidrule(lr){4-5} \cmidrule(lr){6-7} % Rule spanning columns 2-3 and 4-5
		
		% Header Row 2: Sub-headers
		{} & \textbf{RZ} & \textbf{SA} & \textbf{RZ} & \textbf{SA} & \textbf{RZ} & \textbf{SA}  \\
		\midrule % Middle rule
		
		% Data Rows
		3h & 2017 & \textbf{4061} & 3823 & \textbf{1439} & \textbf{243} & 308 \\
		6h & 2859 & \textbf{4157} & 4061 & \textbf{1393} & \textbf{266} & 312 \\ 
		Whole Day & \textbf{9888} & 6452 & {8122} & \textbf{1688} & 525 & \textbf{344} \\
		\bottomrule % Bottom rule
	\end{tabularx}
\end{table}

\begin{figure}[h]
	\centering
	
	% --- First Subfigure ---
	\begin{subfigure}[b]{0.49\linewidth} % The width is slightly less than half to allow for spacing
		\centering
		\includegraphics[width=\textwidth]{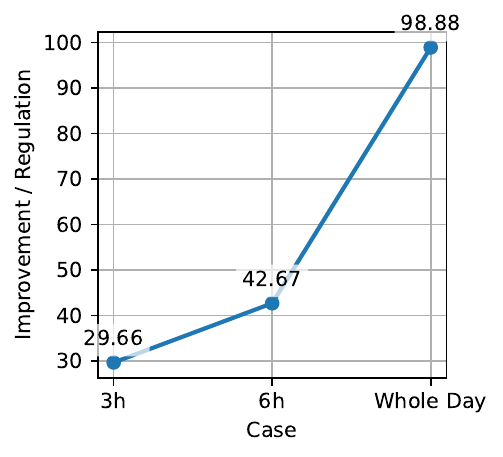}
		\caption{Efficiency per Regulation.}
		\label{fig:rzsaefficiencyperregulation}
	\end{subfigure}
	\hfill % Adds horizontal space between the figures
	%
	% --- Second Subfigure ---
	\begin{subfigure}[b]{0.49\linewidth} % The width is slightly less than half
		\centering
		\includegraphics[width=\textwidth]{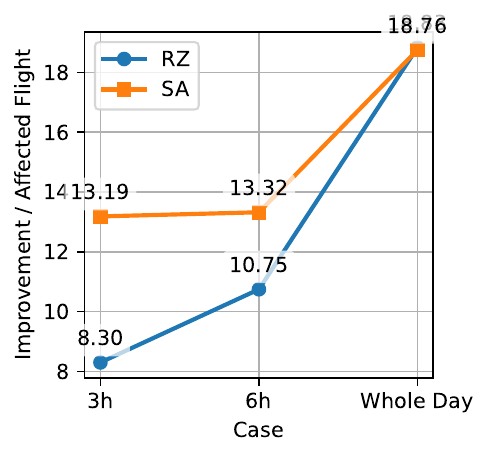}
		\caption{Efficiency per Affected Flight.}
		\label{fig:rzsaefficiencyperaffectedflight}
	\end{subfigure}
	
	\caption{Efficiency Metrics of RZ and SA.} % Main caption for the combined figure
	\label{fig:combined_efficiency_comparisons}
\end{figure}

RegulationZero and SA's problem solving abilities can be summarized in Table \ref{tab:rzsa} with additional performance information provided in Table \ref{tab:rz_meta} and Table \ref{tab:sa_meta}. As the planning horizon expands, RZ demonstrates pronounced efficiency at scale. Its improvement per minute of delay more than doubles (0.53 to 0.70, and then to 1.22 for whole-day scale), while the improvement per regulation increases roughly 3.3 times (from roughtly 30 to 43, and then to 99). At the full-day horizon, RZ achieves near parity in{improvement per affected flight with SA baseline: 18.83 vs.~18.76, despite the additional FPFS constraint (Figure \ref{fig:combined_efficiency_comparisons}). This indicates that when RZ does impose delays, their marginal utility is comparable to those generated by SA.

A particularly distinctive trait of RZ is its selectivity: only about {20 to 25\%} of flights involved in its regulations experience actual delays (20.4\%, 23.0\%, 24.5\%), meaning the majority of flights touched by the regulation plan remain on schedule. For shorter horizons (3h/6h), RZ delays {21\%} and {15\%} fewer flights than SA (243 vs.~308; 266 vs.~312) while still securing substantial objective gains.

\begin{figure}[h!]
	\centering
	\includegraphics[width=0.7\linewidth]{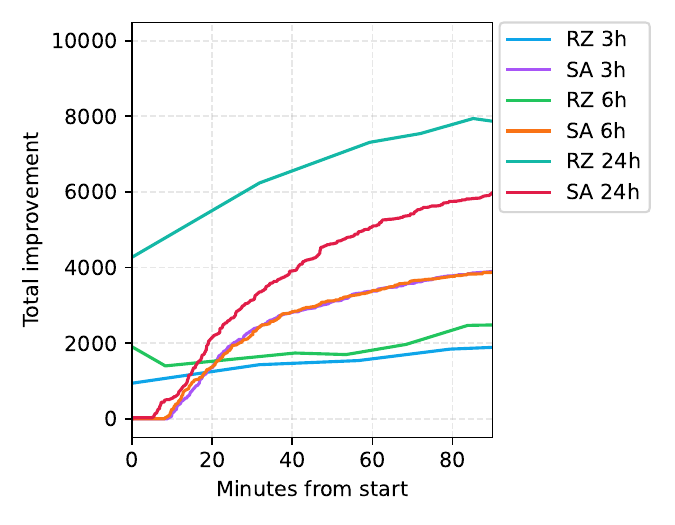}
	\caption{Evolution of Objective Improvements with time for both RZ and SA, capped to a 90-minute compute budget. SA curves observed delay due to heat-up phase.}
	\label{fig:intraprogresscomparison}
\end{figure}

Figure~\ref{fig:intraprogresscomparison} illustrates the evolution of solution quality over time. The results clearly indicate that RZ exhibits weaker performance on smaller-scale problems, where SA dominates once the compute time exceeds approximately 20 minutes. In contrast, for large-scale instances, RZ consistently outperforms SA across the entire runtime horizon, demonstrating strong anytime characteristics. This property makes RZ appealing for tactical deployment scenarios, where large-scale problems must be resolved frequently under stringent computational time constraints, or very large scale problem that needs to be solved pre-tactically.

\subsection{Hyperparameter and Ablation Study}
As RegulationZero involves a multitude of tunable parameters, we concentrate on a subset of hyperparameters that exert the most significant influence on its performance. Consistent with standard RL frameworks, our analysis emphasizes those parameters that govern the trade-off between exploration and exploitation.

\subsubsection{Number of Hotspots}
To assess the impact of constraining RegulationZero to the top-$k$ hotspots ranked by severity, we execute the algorithm using the hyperparameter configurations listed in Table~\ref{tab:hyperparams}, with max\_hotspots\_per\_node set to 3, 6, and 12, respectively. The corresponding benchmark results are summarized in Table~\ref{tab:rz_hotspots}. It is worth noting that the outcomes for the 12-hotspot configuration may exhibit slight variations relative to those reported in the preceding section, owing to the inherent stochasticity of the algorithm.

\begin{table}[htbp]
	\centering % Center the table on the page
	\caption{Comparison of Results for 3, 6, and 12 hotspots allowed to be selected by RZ.}
	\label{tab:rz_hotspots}
	% Define 5 columns: l (left-aligned) for Case, and four X (expanding, centered) columns.
	\begin{tabularx}{\textwidth}{l *{4}{>{\centering\arraybackslash}X}}
		\toprule % Top rule
		% Header Row 1: Merged Columns
		\textbf{Case} & \textbf{Final Objective Improvements} & \textbf{Total Delay (min)} & \textbf{Number of Regulations} & \textbf{Number of Flights} \\
		
		\midrule
		
		% Data Rows
		3 hotspots & 817 & 5063 & 24 & 324 \\
		6 hotspots & 4929 & 7141 & 61 & 448 \\
		12 hotspots & 11264 & 7686 & 82 & 470 \\
		\bottomrule % Bottom rule
	\end{tabularx}
\end{table}

\begin{figure}[h]
	\centering
	
	% --- First Subfigure ---
	\begin{subfigure}[b]{0.49\linewidth} % The width is slightly less than half to allow for spacing
		\centering
		\includegraphics[width=\textwidth]{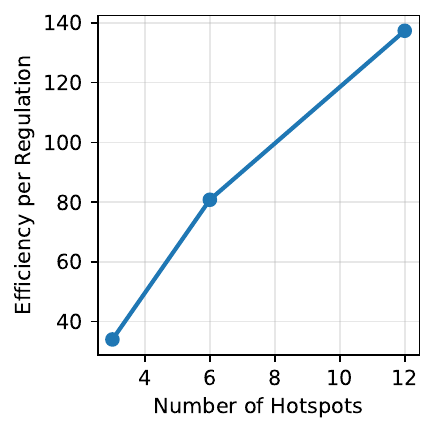}
		\caption{Efficiency per Regulation.}
		\label{fig:hsp_efficiency_per_regulation}
	\end{subfigure}
	\hfill % Adds horizontal space between the figures
	%
	% --- Second Subfigure ---
	\begin{subfigure}[b]{0.49\linewidth} % The width is slightly less than half
		\centering
		\includegraphics[width=\textwidth]{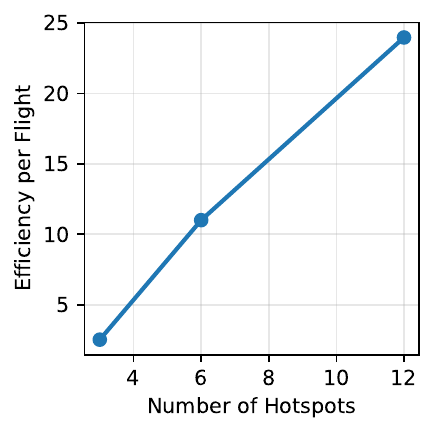}
		\caption{Efficiency per Affected Flight.}
		\label{fig:hsp_efficiency_per_flight}
	\end{subfigure}
	
	\caption{Efficiency Metrics' Relationship with Number of Hotspots Allowed to be Discovered by RZ.} % Main caption for the combined figure
	\label{fig:hsp_efficiency}
\end{figure}

The results indicate that restricting RegulationZero to only the most severe hotspots substantially diminishes its capacity to resolve DCB imbalances. Despite the modest expansion in the number of delayed flights, the overall improvement more than doubles suggesting that a broader hotspot set enables the discovery of high-leverage interventions. This trend is further reflected in the reduction of flights regulated per action, from 13.5 to 5.73 indicating that the resulting regulations become more targeted and efficient. These results point in the same direction as noted by \cite{Dalmau2021Indispensable} that only a few regulations possess significant leverage on the DCB situation, and these regulations do not usually coincide with the most severe hotspots.

\subsubsection{PUCT Term and Number of Simulations}
These two hyperparameters control the amount of exploration. In our first experiment batch, we keep the number of simulations constant, while increasing PUCT. The results are summarized in Table \ref{tab:rz_puct}.

\begin{table}[htbp]
	\centering % Center the table on the page
	\caption{Comparison of Results for different PUCT settings.}
	\label{tab:rz_puct}
	% Define 5 columns: l (left-aligned) for Case, and four X (expanding, centered) columns.
	\begin{tabularx}{\textwidth}{l *{4}{>{\centering\arraybackslash}X}}
		\toprule % Top rule
		% Header Row 1: Merged Columns
		\textbf{Case} & \textbf{Final Objective Improvements} & \textbf{Total Delay (min)} & \textbf{Number of Regulations} & \textbf{Number of Flights} \\
		
		\midrule
		
		% Data Rows
		PUCT=32+SIMS=12 & 10899 & 8621 & 85 & 534 \\
		PUCT=64+SIMS=12 & 10779 & 9191 & 96 & 584 \\
		PUCT=128+SIMS=12 & 9194 & 8356 & 83 & 554 \\
		\midrule
		PUCT=32+SIMS=4 & 9233 & 9317 & 89 & 580 \\
		PUCT=64+SIMS=8 & 9901 & 8589 & 79 & 545 \\
		PUCT=128+SIMS=16 & 10008 & 7972 & 83 & 480 \\
		\bottomrule % Bottom rule
	\end{tabularx}
\end{table}

\begin{figure}[h]
	\centering
	
	% --- First Subfigure ---
	\begin{subfigure}[b]{0.49\linewidth} % The width is slightly less than half to allow for spacing
		\centering
		\includegraphics[width=\textwidth]{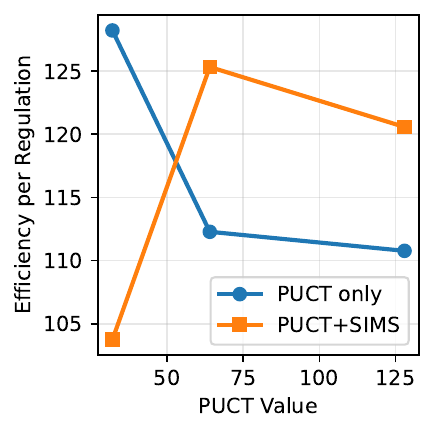}
		\caption{Efficiency per Regulation.}
		\label{fig:puct_efficiency_per_regulation}
	\end{subfigure}
	\hfill % Adds horizontal space between the figures
	%
	% --- Second Subfigure ---
	\begin{subfigure}[b]{0.49\linewidth} % The width is slightly less than half
		\centering
		\includegraphics[width=\textwidth]{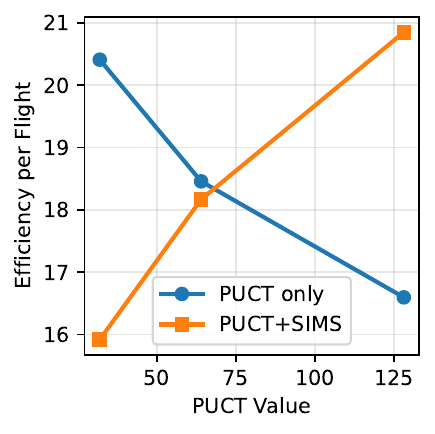}
		\caption{Efficiency per Affected Flight.}
		\label{fig:puct_efficiency_per_flight}
	\end{subfigure}
	
	\caption{Efficiency Metrics' Relationship with PUCT and SIM Values Used by RZ.} % Main caption for the combined figure
	\label{fig:puct_efficiency}
\end{figure}

Although the initial outcomes may appear counterintuitive, they can be explained by the randomized nature of exploration in our implementation. Specifically, increasing the PUCT term without proportionally expanding the number of simulations amplifies the variability of rewards, thereby impairing the learning efficiency of RZ. This effect manifests as a monotonic decline in performance across all evaluated metrics, including both the objective value and regulatory efficiency (Figure \ref{fig:puct_efficiency}).

When the increase in the PUCT term is accompanied by an increasing number of simulations, a partial restoration of exploration efficacy can be observed. The additional simulations yield more reliable reward signals, which reflected through the improving trend illustrated in Figures~\ref{fig:puct_efficiency_per_regulation} and~\ref{fig:puct_efficiency_per_flight}.

\subsubsection{Impact of Flow Scoring Factors}
In this section, we study the sensitivity of the two heuristic factors built to guide RZ to select the most promising flows: flow pressure, and the flow's slack. We derived 4 cases and summarized the results in Table \ref{tab:rz_flow_feature}.

\begin{table}[htbp]
	\centering % Center the table on the page
	\caption{Comparison of Results for Different Flow Feature Weights. Note that all weights are not normalized.}
	\label{tab:rz_flow_feature}
	% Define 5 columns: l (left-aligned) for Case, and four X (expanding, centered) columns.
	\begin{tabularx}{\textwidth}{l *{4}{>{\centering\arraybackslash}X}}
		\toprule % Top rule
		% Header Row 1: Merged Columns
		\textbf{Case} & \textbf{Final Objective Improvements} & \textbf{Total Delay (min)} & \textbf{Number of Regulations} & \textbf{Number of Flights} \\
		
		\midrule
		
		% Data Rows
		P=4, S15=0, S30=0 & 9194 & 8756 & 100 & 524 \\
		P=4, S15=0.25, S30=0 & 10730 & 9860 & 93 & 521 \\
		P=6, S15=0.25, S30=0.25 & 11608 & 8842 & 79 & 514 \\
		P=8, S15=0.5, S30=0.5 & 11191 & 8939 & 82 & 512 \\
		
		\bottomrule % Bottom rule
	\end{tabularx}
\end{table}

The overall trends indicate that all three factors are critical for selecting high-quality flows, as the inclusion \textit{Slack15}, and \textit{Slack30} jointly enhances performance by approximately 1500 and 900 points, respectively. The configuration $(P=8, S15=0.5, S30=0.5)$ further reveals that among these features, flow pressure exerts the most pronounced influence: increasing the relative weight of the slack-based features leads to a deterioration in the final results. These findings show that tuning these hyperparameters is a complex process, and we hypothesize that the optimal configuration is likely problem-dependent.

\subsection{Performance Robustness}\label{subsec:perf_robustness}
To examine the replicability of RZ's scalability on larger problem instances, we re-executed RZ under identical settings for two additional operational days of July 18 and 29, 2023. It is important to note that, due to substantial variation in traffic patterns and the number of scheduled flights across these dates, absolute objective improvements are not directly comparable between instances.

\begin{table}[htbp]
	\centering % Center the table on the page
	\caption{Comparison of Results for Different Traffic Dates.}
	\label{tab:rz_replica}
	% Define 5 columns: l (left-aligned) for Case, and four X (expanding, centered) columns.
	\begin{tabularx}{\textwidth}{l *{5}{>{\centering\arraybackslash}X}}
		\toprule % Top rule
		% Header Row 1: Merged Columns
		\textbf{Case} & \textbf{Baseline Objective} & \textbf{Final Objective Improvements} & \textbf{Total Delay (min)} & \textbf{Number of Regulations} & \textbf{Number of Flights} \\
		
		\midrule
		
		% Data Rows
		18/07 3h & 62610 & 23870 & 14060 & 116 & 564 \\
		18/07 6h & 127770 & 40676 & 17224 & 125 & 650 \\
		18/07 24h & 567970 & 79972 & 26308 & 111 & 886 \\
		
		\midrule 
		
		29/07 3h & 151720 & 44299 & 24881 & 128 & 755 \\
		29/07 6h & 240120 & 70583 & 31597 & 137 & 931 \\
		29/07 24h & 758100 & 110520 & 33590 & 106 & 1187 \\
		\bottomrule % Bottom rule
	\end{tabularx}
\end{table}

\begin{figure}
	\centering
	\includegraphics[width=\linewidth]{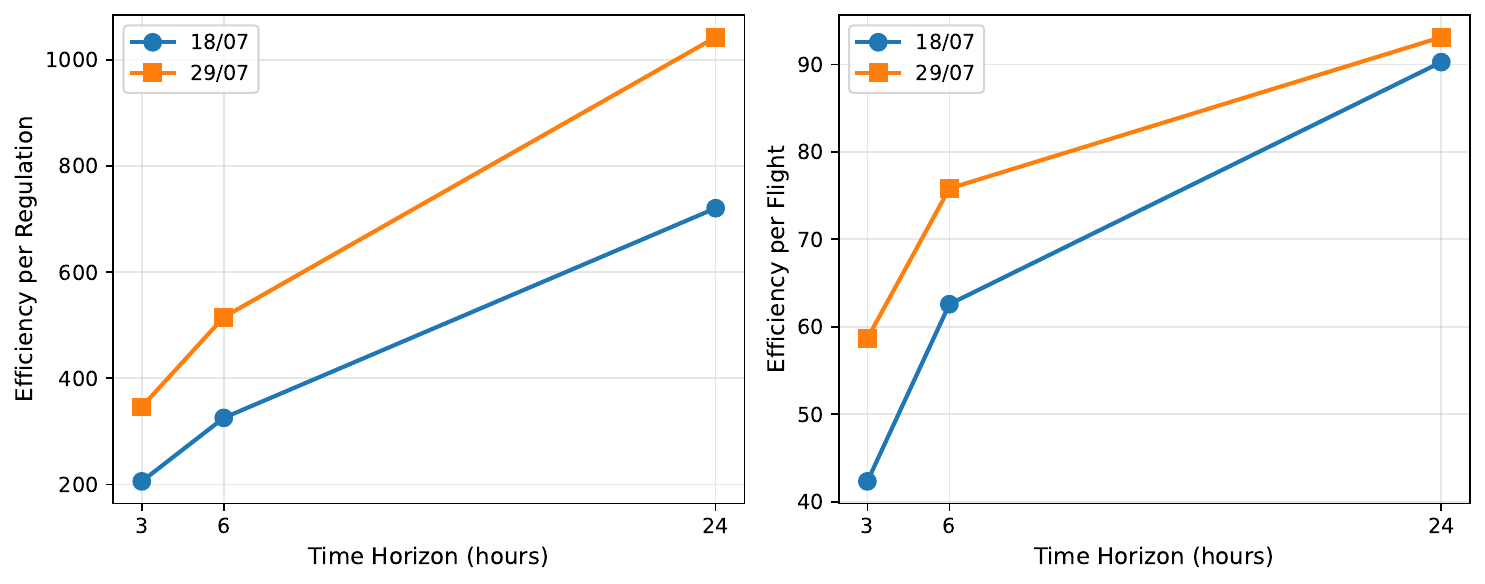}
	\caption{Comparisons between different dates on Efficiency Metrics.}
	\label{fig:multidateefficiencycomparison}
\end{figure}

Table~\ref{tab:rz_replica} and Figure~\ref{fig:multidateefficiencycomparison} confirm that RZ’s scaling efficiency remains consistent across different operational days: both regulatory effectiveness and per-flight delay performance improve with longer planning horizons. Nonetheless, substantial variation in absolute performance is evident between dates, thus prompted our closer inspection on the Cortex platform. We note that not only the DCB imbalance was more severe on July~29, but congestion was notably concentrated within large TFVs overseeing major regions as well. As a prominent example: TFV~LFMWMFDZ, which manages traffic in the  South East of France was overloaded (Figure \ref{fig:lfmw}), leading to a break down of flow pattern extraction in the community detection module. While this is not a failure per-se, this example highlights the need to adapt the graph edge creation threshold value to specific circumstances. A detailed analysis of how such spatial-structural traffic characteristics modulate RZ’s performance is beyond the scope of this paper but remains an important direction for future research.

\begin{figure}
	\centering
	\includegraphics[width=\linewidth]{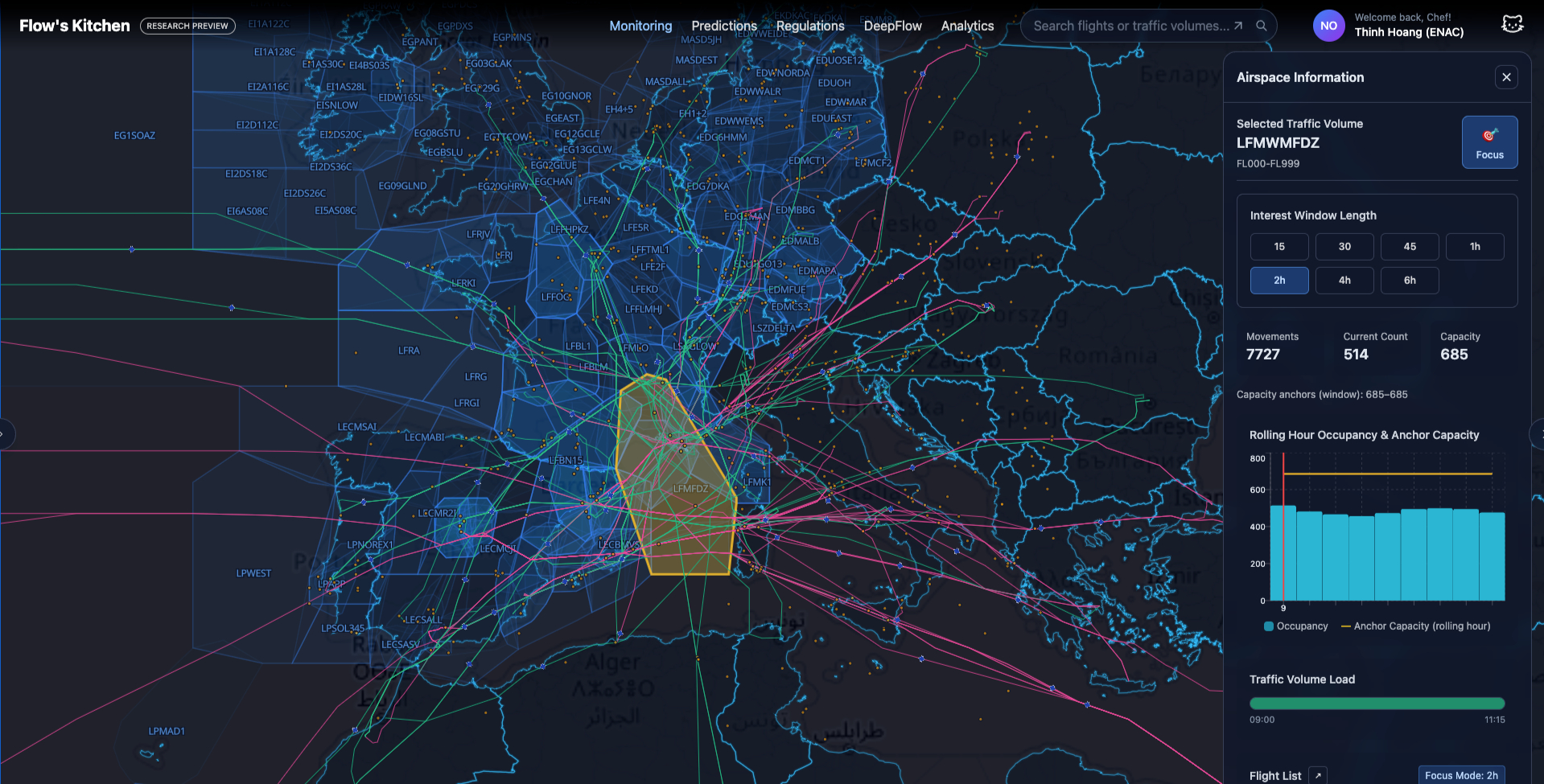}
	\caption{The overloaded traffic volume LFMWMFDZ captures flights that do not have a clear "flow" pattern, thus RZ decided to impose a blanket delay, affecting lots of flights, but unlocked a major exceedance relief.}
	\label{fig:lfmw}
\end{figure}

\subsection{Benchmark against Greedy Policy}
Table \ref{tab:rz_v_greedy} presents a benchmark comparison between RegulationZero and a greedy regulation policy that enforces uniform restrictions by capping the entry rate to the capacity threshold defined by the TFVs. For fairness, the greedy policy was constrained to a maximum of 128 regulations, comparable to RZ's limit of 64 commits, which generally yields between 80 and 120 individual regulations. The comparative outcomes reveal consistently negative improvements across all evaluated instances. This pattern revealed that in absence of a well-formulated regulation strategy, simplistic or blanket interventions merely relocate congestion while paying the cost for delay rather than mitigating it, making this a classic case of Regulation Cascading.

\begin{table}[htbp]
	\centering % Center the table on the page
	\caption{Comparison of Results between RZ and Greedy Policy}
	\label{tab:rz_v_greedy}
	% Define 5 columns: l (left-aligned) for Case, and four X (expanding, centered) columns.
	\begin{tabularx}{\textwidth}{l *{4}{>{\centering\arraybackslash}X}}
		\toprule % Top rule
		% Header Row 1: Merged Columns
		\textbf{Case} & \textbf{RZ's Final Objective Improvements} & \textbf{Total Delay (min)} & \textbf{GP's Finial Objective Improvements} & \textbf{Total Delay (min)} \\
		
		\midrule
		
		% Data Rows
		18/07 & 79972 & 26308 & -1656626 & 204446 \\
		20/07 & 9888 & 8122 & -1330512 & 178462 \\
		29/07 & 110520 & 33590 & -773837 & 163417 \\		
		\bottomrule % Bottom rule
	\end{tabularx}
\end{table}

\section{Discussion}
In this section, we interpret the significance of the results obtained thus far and outline a structured roadmap for advancing RZ toward operational readiness.

\subsection{FPFS/Regulations could still be moderately effective for very large scale problem with constrained computation time budget, given that the regulation plan is sound}
Although it is well established in the literature that slot optimization methods such as SA possseses a substantial advantage in delay efficiency, and our results in fact corroborate this pattern, we find it noteworthy that, with well-structured regulation planning, RZ achieves performance comparable to, and in the large-scale case of whole-day planning, exceeding that of SA. These trends suggest that RZ’s extended regulation sequences can yield higher objective returns at scale, providing greater leverage both per regulation and per affected flight, whereas SA continues to excel in minimizing total delay across problem sizes. The observed convergence in \textit{benefit per delayed flight} for the full-day scenario implies that, once the system grows sufficiently large, the structure of delays (as captured by RZ’s sequential planning) can offset the quantity or magnitude of delays (as optimized by SA). Taken together, these findings appear to align with the spirit of the No-Free-Lunch theorem: no single method consistently dominates across objectives or scales, even within instances characteristic of ATFM. We believe this is an open and important question of the field, and more research is needed to yield a conclusive answer.

\subsection{Realism}
While RZ's solutions are designed for compatibility with existing slot allocation algorithms, several discrepancies remain that must be addressed before the system can achieve certifiable status. Most important among these is the validation of the FIFO queue slot allocation against CASA and RBS++. Another important concern pertains to the treatment of demand and capacity profiles. In typical operational settings, these profiles are defined on a rolling-hour basis with a 20-minute stride; operators, however, have the flexibility to modify these parameters to better suit the temporal resolution of the problem at hand. RZ, by contrast, is constrained to a fixed calculation setting, which limits its adaptability for a wide range of DCB scenarios.

Furthermore, the preliminary identification of TFVs to regulate within RZ is restricted to hotspots only. While this design choice effectively constrains the search space and improves computational tractability, it simultaneously narrows the spectrum of regulatory options during the exploration phase. Finally, the framework's long-term value therefore depends heavily on its ability to extend toward a warm-start capability. Such functionality would enable RZ to operate more effectively in pre-tactical and tactical phases, where CASA/RBS++ CTOT recalculation alone is insufficient. These are non-trivial challenges that warrant substantial future research. Another direction is to extend RZ to plan in face of uncertainty in demands and weather scenarios. Without these enhancements, RZ remains primarily limited to strategic or early pre-tactical planning horizons. 

\subsection{Component Dependency and Expert Knowledge Injection}
RZ contains many constituting subcomponents, each governed by numerous hyperparameters that influence performance in profound ways. Identifying parameter configurations that adapt seamlessly to diverse operational scenarios therefore represents a complex challenge. As observed in Section~\ref{subsec:perf_robustness}, certain scenarios might break the community detection algorithm under dense traffic with highly heterogeneous directions. In principle, dynamically adjusting threshold values for larger TFVs could restore the surgical regulatory moves that RZ routinely demonstrates on smaller TFVs. However, it is not yet known at this moment how it will be achieved concretely.

Despite these challenges, RZ's design offers many promising research directions. Its decision logic operates in the same conceptual space as human operators, making it uniquely amenable to the direct infusion of expert judgment. This capability allows practitioners to embed tacit operational knowledge such as local traffic management practices or context-specific routing rules directly into the tree expansion process. This is in contrast with other methods such as the Mixed-Integer Linear Programming (MILP) formulation, which is often demanding about mathematical problem design and relaxations. We believe this integration of human expertise represents one of the most promising frontiers to be discovered, and that RZ could be a stepping stone in realizing a certifiable decision support system in ATFCM.

\section{Conclusion}
In this paper, we introduced RegulationZero, a sequential planning framework purpose-built to model and manage interacting regulations and their order-dependent cascading effects. By shifting optimization into the regulation space and coupling a hierarchical Monte Carlo Tree Search (MCTS) with a structured proposal engine, RegulationZero produces solutions that remain fully compatible with current and emerging slot-allocation processes (e.g., CASA, RBS++, and Targeted CASA), thereby minimizing integration risk. Crucially, the search operates within the same conceptual domain as flow management experts, making the decision process both transparent and intuitively aligned with operational practice.

Comprehensive experiments on large-scale European traffic confirm that RegulationZero's performance is promising on large-scale problems. It exhibits strong anytime properties, achieves higher objective improvements than a trajectory-space simulated-annealing baseline over extended horizons over a reasonable planning time, and remains selective, impacting only a small portion of flights. The approach generalizes across variable traffic conditions and decisively outperforms a naive rate-capping policy that exhibits the characteristics of regulation cascading. Ablation studies further highlight the contribution of structured exploration and of the flow heuristics in discovering high-leverage regulations.

Several avenues for further research remain open, including rigorous validation against operational systems and real-world data, systematic integration of expert knowledge into the planning process, and a deeper investigation of hyperparameter sensitivities to better exploit inherent traffic patterns.

\appendix
\section{RegulationZero as Regularized Policy Optimization}\label{sec:poliopt}

This appendix demonstrates that the formulation of RegulationZero introduced in Section~\ref{sec:method} can be interpreted as a form of \emph{regularized policy optimization} that jointly employs forward and reverse Kullback-Leibler (KL) divergences to shape both hotspot and regulation proposal selections. For tractability, we focus on the simplified setting where \emph{only one hotspot/expansion attempt is associated with each node in the search tree}, deferring extensions for future research.

\subsection{RegulationZero attempts to track a policy regularized by a forward KL for hotspot selection and reverse KL for proposal selection}
If we define the factorized policy $\pi(h,r|s)=\pi_H(h|s)\pi(r|h,s)$ over the two priors, and the corresponding regularization term as:
\begin{equation}
    \Omega_s(\pi) = \tau_h \kl{\pi_H(\cdot | s)}{P_H(\cdot | s)}+ \expect{h \sim \pi_H}{ \lambda_R (s,h) \kl{P_R(\cdot|h,s)}{\pi_R(\cdot|h,s))}}
\end{equation}

Define a general vector of state backup values $q = q(s,h,r)$, we define the \textit{regularized objective} as:
\begin{equation}\label{eq:gs_problem}
    G_s(q) = \sup_\pi [\expect{h,r \sim \pi}{q(s,h,r) - \Omega_s(\pi)}]
\end{equation}

The following Proposition shows that the solution to $G_s(q)$ has a decomposition form:

\begin{proposition}\label{prop:pi_ref}
For each \(h\), let \(q_h(\cdot) = q(s,h,\cdot)\) and \(\lambda_h = \lambda_R(s,h)\). The inner maximizer \(\bar{\pi}_R(\cdot|h,s)\) and its attained value \(U_h(s)\) solve
\[
\bar{\pi}_R(r|h,s) = \frac{\lambda_h \, P_R(r|h,s)}{\alpha_h - q_h(r)},
\quad
\text{with } \alpha_h > \max_r q_h(r) \text{ chosen s.t. } \sum_r \bar{\pi}_R(r|h,s) = 1,
\]
\[
U_h(s) = \sup_{y} \left[ q_h \cdot y - \lambda_h \, \kl{P_R}{y} \right] = \alpha_h - \lambda_h - \lambda_h \sum_r P_R(r|h,s) \log\!\left(\frac{\alpha_h - q_h(r)}{\lambda_h}\right).
\]
With inclusive inner values \(U_h(s)\),
\[
\bar{\pi}_H(h|s) \propto P_H(h|s) \exp\!\left(U_h(s) / \tau_h\right),
\quad
\text{and}
\quad
G_s(q) = \tau_h \log \sum_h P_H(h|s) \exp\!\left(U_h(s) / \tau_h\right).
\]
Thus \(\bar{\pi}\) is the nested solution
\[
\bar{\pi}_H(h|s) \propto P_H(h|s) \exp\!\left(U_h / \tau_h\right),
\qquad
\bar{\pi}_R(\cdot|h,s) \text{ as above.}
\]
\end{proposition}

\begin{proof}
We begin by expanding $G_s(q)$:

\begin{align*}
    G_s(q) &= \sup_\pi [\expect{h,r \sim \pi}{q(s,h,r)] - \Omega_s(\pi)} = \sup_\pi [\sum_{h, r} \pi_H(h|s)\pi_R(r|h,s) q(s,h,r) - \Omega_s(\pi)]\\
    &= \sup_\pi [ \sum_{h, r} \pi_H(h|s) \pi_R(r|h,s) q(s,h,r) - \tau_h \kl{\pi_H(\cdot | s)}{P_H(\cdot | s)} \\
    &- \sum_h \pi_H(h|s) \lambda_R (s,h) \kl{P_R(\cdot|h,s)}{\pi_R(\cdot|h,s))}] \\
    &= \sup_\pi \bigg[ \sum_h \pi_H(h|s) [\sum_r \pi_R(r|h,s) q(s,h,r) - \lambda_R (s,h) \kl{P_R(\cdot|h,s)}{\pi_R(\cdot|h,s))}] \\
    &- \tau_h \kl{\pi_H(\cdot | s)}{P_H(\cdot | s)} \bigg] \\
    &= \sup_{\pi_H} \bigg[ \sum_h \pi_H(h|s) \sup_{\pi_R} [\sum_r \pi_R(r|h,s) q(s,h,r) - \lambda_R (s,h) \kl{P_R(\cdot|h,s)}{\pi_R(\cdot|h,s))}] \\
    &- \tau_h \kl{\pi_H(\cdot | s)}{P_H(\cdot | s)} \bigg]
\end{align*}

We define the inner problem:
$U_h(q_h) = \sup_{y} [q_h^\intercal y - \lambda_R (s,h) \kl{P_R(\cdot|h,s)}{y}]$. Given $\lambda_R (s,h) > 0$, the problem is convex in y, and given the constraint for $y$ to be a valid probability distribution, we could define the Lagrangian:
\begin{equation}
\mathcal{L} = q_h^\intercal y - \lambda_R \kl{P_R}{y} - \alpha (y_r - 1)
\end{equation}

The stationary condition gives:
\begin{equation}\label{eq:inner_loop_sol}
    y^* = \frac{\lambda P_R}{\alpha - q_h}
\end{equation}
and $\alpha$ must solve for:
\begin{equation}
    \sum_r \frac{\lambda P_R(r)}{\alpha - q_h(r)} = 1
\end{equation}
in order to normalize the distribution $y = \pi_R$. The attained value corresponding to $y^*$ is:
\[
U_h(q_h) = q_h^\intercal y^* - \lambda_R \sum_r P_R(r) \log \frac{P_R(r)}{y^*(r)} = \alpha - \lambda_R - \lambda_R \sum_r P_R(r) \log \bigg(\frac{\alpha - q_r}{\lambda}\bigg)
\]

We notice that this expression has the same form as in \cite{grill2020monte}. Now, given that the inner problem attains its maximum at \( U_h(q_h) \) as derived above, the outer optimization problem can be expressed in the following closed form:

\begin{equation}
    \sup_{x \in \Delta(\mathcal{H})} \; x^\top U(q_x) - \tau_h \, \mathrm{KL}(x \, \| \, P_H),
\end{equation}

where \( \Delta(\mathcal{H}) \) represents the probability simplex over \(\mathcal{H}\).

Because this corresponds to the classical maximum-entropy regularized policy optimization formulation \cite{ziebart2010modeling}, the optimal solution admits the familiar log-sum-exp (softmax) form:

\begin{equation}
    x_h^* \propto P_H(h) \, \exp\!\left(\frac{U_h}{\tau_h}\right),
    \qquad
    x^*(h) = 
    \frac{P_H(h) \, \exp\!\left(U_h / \tau_h\right)}
         {\sum_{h' \in \mathcal{H}} P_H(h') \, \exp\!\left(U_{h'} / \tau_h\right)}
\end{equation}
\end{proof}

Proposition \ref{prop:pi_ref} says that the factored policy $\bar{\pi} = \bar{\pi}_R\bar{\pi}_H$ solves the regularized objective maximizing problem $G_s(q)$. The PUCT selection rule (\ref{eq:puct_selection}) was shown in \cite{grill2020monte} to be equal to:
\begin{equation}
    a^* = \arg \max_a \frac{\partial}{\partial n(a)} [q^\intercal y - \lambda_N \kl{\pi_\theta}{\hat{\pi}}]
\end{equation}
which can be viewed as a gradient descent step towards the optimizer of the regularized problem:
\begin{equation}
    \bar{\pi} = q^\intercal y - \lambda_N \kl{\pi_\theta}{\hat{\pi}}
\end{equation}
which in turn has the same solution form $y = \lambda_N \pi_\theta / (\alpha - q)$ as (\ref{eq:inner_loop_sol}). In other words, the inner loop's MCTS attempts to track the policy:
\begin{equation}
    \sup_{y} [q_h^\intercal y - \lambda_R (s,h) \kl{P_R(\cdot|h,s)}{y}]
\end{equation}

Combined with the outer loop's Boltzmann distribution's form, we can conclude that jointly, the RegulationZero framework tracks the solution for the problem $G_s $(\ref{eq:gs_problem}).

\begin{equation*}
    \boxed{\hat{\pi} \approx \bar{\pi} = \arg\max_\pi [\expect{h,r \sim \pi}{q(s,h,r) - \Omega_s(\pi)}]},
\end{equation*}
if the outer loop's hotspot severity proxy function is assumed to track
\begin{equation}
	\phi(s) \approx U(s) = \sup_{y} \left[ q_h \cdot y - \lambda_h \, \kl{P_R}{y} \right].
\end{equation}

\subsection{Sample Complexity Estimation}
Let \(N\) be the total number of root simulations and let \(n_H(h)\) denote the number of times the hotspot
\(h\) is selected at the root, so that \(\sum_{h} n_H(h)=N\).  
We write \(\hat\pi_H,\hat\pi_R\) for RegulationZero respectively.

For any pair \((h,r)\) we have
\begin{align*}
    \bigl|\hat\pi(h,r)-\bar\pi(h,r)\bigr|
        &=\bigl|\hat\pi_H(h)\,\hat\pi_R(r\mid h)-\bar\pi_H(h)\,\bar\pi_R(r\mid h)\bigr|\\
        &\le \bigl|\hat\pi_H(h)-\bar\pi_H(h)\bigr|\;\hat\pi_R(r\mid h)
           +\bar\pi_H(h)\,\bigl|\hat\pi_R(r\mid h)-\bar\pi_R(r\mid h)\bigr|\\
        &\le \bigl|\hat\pi_H(h)-\bar\pi_H(h)\bigr|
           +\bigl|\hat\pi_R(r\mid h)-\bar\pi_R(r\mid h)\bigr|.
\end{align*}
Taking the supremum over all \((h,r)\) yields the key reduction
\begin{equation}
    \|\hat\pi-\bar\pi\|_{\infty}
        \;\le\; \|\hat\pi_H-\bar\pi_H\|_{\infty}
               + \max_{h}\,\|\hat\pi_R(\cdot\mid h)-\bar\pi_R(\cdot\mid h)\|_{\infty}.
    \label{eq:joint-reduction}
\end{equation}

For each hotspot \(h\), assuming that \(\bar\pi_R(\cdot\mid h)\) is locally constant as counts grow,
the reversed-KL tracking bound gives
\begin{equation}
    \|\hat\pi_R(\cdot\mid h)-\bar\pi_R(\cdot\mid h)\|_{\infty}
        \;\le\; \frac{|R|-1}{\,|R|+n_H(h)}.
    \label{eq:inner-bound}
\end{equation}

If the hotspot at each root selection is sampled from \(\bar\pi_H\),
Hoeffding’s inequality together with a union bound implies that, with probability at least
\(1-\delta\),
\begin{equation}
    \|\hat\pi_H-\bar\pi_H\|_{\infty}
        \;\le\; \sqrt{\frac{1}{2N}\,\log\!\frac{2|H|}{\delta}}.
    \label{eq:outer-bound}
\end{equation}

Hence, for any run with total simulations \(N\) and realised gate counts \(n_H(h)\), we obtain,
with probability at least \(1-\delta\),
\begin{equation}
    \boxed{\|\hat\pi-\bar\pi\|_{\infty}
        \;\le\;
          \sqrt{\frac{1}{2N}\,\log\!\frac{2|H|}{\delta}}
          \;+\;
          \max_{h}\frac{|R|-1}{\,|R|+n_H(h)}}
    \label{eq:anytime-bound}
\end{equation}

\section{Sample Regulation Solution Trajectory}
Figure \ref{fig:regulation_trajectory} illustrates a representative regulation trajectory generated by RegulationZero. For clarity, we report only the number of flights rather than the complete set of flight identifiers, and omit the specific time window. Detailed information can be found in the supplementary log files. The top row denotes the identifier of the control volume (or reference location). A zero entry rate indicates that all inbound traffic is held within the window, leading to a bunching effect toward the end of the regulatory period.

% Required packages in your LaTeX preamble:
% \usepackage{tikz}
% \usepackage{xcolor}
% \usetikzlibrary{positioning,arrows.meta}

% Regulation Solution Trajectory
% Total Improvement: 2017.0
% Trajectory: EDUFZ44:44-46[0] -> EIDW16SL:48-51[0] -> ENS347S:44-48[0] -> LFMWW:51-51[2] -> LSAZM15:50-51[0] -> ESOSF:44-47[0] -> LEMGANM:45-47[1] -> EI2D216C:48-51[0] -> EISNALL:49-51[0] -> LECMCJU:44-47[0] -> LECSCEN:46-47[0] -> EDMSI:44-47[0] -> EBBUNLW1:40-43[0] -> EDG5PFEI:46-47[0] -> LFFARML1:48-51[0] -> EDMSWA:48-48[0] -> LFFFPAE:42-44[0] -> EI1A128C:47-47[1] -> LECMZUMX:48-48[0] -> EGCLW:48-49[0] -> LECPDWO:51-51[0] -> LECBGO1:48-48[0] -> ESMM8:41-44[0] -> MASB5OH:46-47[0] -> LEBL23FE:49-50[0] -> ESOS179:46-48[0] -> EDUWUR24:40-42[0] -> LFRGA:40-41[0] -> EDWDBDS:51-51[0] -> EISNALL:40-40[1] -> EDUNTM44:46-47[0] -> LFBP1:44-44[0] -> EDUL234:51-51[0] -> LFRRMSI:44-44[0] -> LECBBSU:42-42[0] -> LFMLOLE:51-51[0] -> ENO34:40-43[0] -> EGCLW:51-51[0] -> LECMZGZ:40-40[0] -> MASD5HH:51-51[0] -> MASD3WL:48-49[1] -> LFBL2:51-51[0]

\begin{figure}[htbp]
\centering
\begin{tikzpicture}[
  node distance=0.3cm and 0.5cm, scale=0.64, transform shape,
  box/.style={
    rectangle,
    draw=black,
    thick,
    minimum width=3.0cm,
    minimum height=1.4cm,
    align=center,
    font=\small,
    fill=blue!10
  },
  arrow/.style={
    ->,
    >=stealth,
    thick,
    shorten >=2pt,
    shorten <=2pt
  }]
  \node[box] (reg0) {\textbf{EDUFZ44} \\\\ Flights: 53 \\\\ Base: 20.0/h \\\\ New: 17.0/h \textcolor{red}{$\downarrow$}};
  \node[box] (reg1) [right=of reg0] {\textbf{EIDW16SL} \\\\ Flights: 54 \\\\ Base: 22.0/h \\\\ New: 15.0/h \textcolor{red}{$\downarrow$}};
  \node[box] (reg2) [right=of reg1] {\textbf{ENS347S} \\\\ Flights: 19 \\\\ Base: 5.6/h \\\\ New: 1.0/h \textcolor{red}{$\downarrow$}};
  \node[box] (reg3) [right=of reg2] {\textbf{LFMWW} \\\\ Flights: 44 \\\\ Base: 4.0/h \\\\ New: 1.0/h \textcolor{red}{$\downarrow$}};
  \node[box] (reg4) [right=of reg3] {\textbf{LSAZM15} \\\\ Flights: 22 \\\\ Base: 14.0/h \\\\ New: 0.0/h \textcolor{red}{$\downarrow$}};
  \node[box] (reg5) [below=of reg0] {\textbf{ESOSF} \\\\ Flights: 5 \\\\ Base: 3.0/h \\\\ New: 0.0/h \textcolor{red}{$\downarrow$}};
  \node[box] (reg6) [right=of reg5] {\textbf{LEMGANM} \\\\ Flights: 28 \\\\ Base: 2.7/h \\\\ New: 2.0/h \textcolor{red}{$\downarrow$}};
  \node[box] (reg7) [right=of reg6] {\textbf{EI2D216C} \\\\ Flights: 26 \\\\ Base: 5.0/h \\\\ New: 1.0/h \textcolor{red}{$\downarrow$}};
  \node[box] (reg8) [right=of reg7] {\textbf{EISNALL} \\\\ Flights: 53 \\\\ Base: 18.7/h \\\\ New: 16.0/h \textcolor{red}{$\downarrow$}};
  \node[box] (reg9) [right=of reg8] {\textbf{LECMCJU} \\\\ Flights: 23 \\\\ Base: 4.0/h \\\\ New: 0.0/h \textcolor{red}{$\downarrow$}};
  \node[box] (reg10) [below=of reg5] {\textbf{LECSCEN} \\\\ Flights: 27 \\\\ Base: 2.0/h \\\\ New: 0.0/h \textcolor{red}{$\downarrow$}};
  \node[box] (reg11) [right=of reg10] {\textbf{EDMSI} \\\\ Flights: 10 \\\\ Base: 2.0/h \\\\ New: 0.0/h \textcolor{red}{$\downarrow$}};
  \node[box] (reg12) [right=of reg11] {\textbf{EBBUNLW1} \\\\ Flights: 25 \\\\ Base: 15.0/h \\\\ New: 15.0/h \textcolor{gray}{$=$}};
  \node[box] (reg13) [right=of reg12] {\textbf{EDG5PFEI} \\\\ Flights: 20 \\\\ Base: 10.0/h \\\\ New: 10.0/h \textcolor{gray}{$=$}};
  \node[box] (reg14) [right=of reg13] {\textbf{LFFARML1} \\\\ Flights: 44 \\\\ Base: 8.0/h \\\\ New: 2.0/h \textcolor{red}{$\downarrow$}};
  \node[box] (reg15) [below=of reg10] {\textbf{EDMSWA} \\\\ Flights: 4 \\\\ Base: 4.0/h \\\\ New: 0.0/h \textcolor{red}{$\downarrow$}};
  \node[box] (reg16) [right=of reg15] {\textbf{LFFFPAE} \\\\ Flights: 76 \\\\ Base: 26.7/h \\\\ New: 24.0/h \textcolor{red}{$\downarrow$}};
  \node[box] (reg17) [right=of reg16] {\textbf{EI1A128C} \\\\ Flights: 23 \\\\ Base: 8.0/h \\\\ New: 8.0/h \textcolor{gray}{$=$}};
  \node[box] (reg18) [right=of reg17] {\textbf{LECMZUMX} \\\\ Flights: 19 \\\\ Base: 8.0/h \\\\ New: 8.0/h \textcolor{gray}{$=$}};
  \node[box] (reg19) [right=of reg18] {\textbf{EGCLW} \\\\ Flights: 31 \\\\ Base: 8.0/h \\\\ New: 5.0/h \textcolor{red}{$\downarrow$}};
  \node[box] (reg20) [below=of reg15] {\textbf{LECPDWO} \\\\ Flights: 37 \\\\ Base: 4.0/h \\\\ New: 4.0/h \textcolor{gray}{$=$}};
  \node[box] (reg21) [right=of reg20] {\textbf{LECBGO1} \\\\ Flights: 16 \\\\ Base: 4.0/h \\\\ New: 2.0/h \textcolor{red}{$\downarrow$}};
  \node[box] (reg22) [right=of reg21] {\textbf{ESMM8} \\\\ Flights: 11 \\\\ Base: 6.0/h \\\\ New: 4.0/h \textcolor{red}{$\downarrow$}};
  \node[box] (reg23) [right=of reg22] {\textbf{MASB5OH} \\\\ Flights: 41 \\\\ Base: 8.0/h \\\\ New: 5.0/h \textcolor{red}{$\downarrow$}};
  \node[box] (reg24) [right=of reg23] {\textbf{LEBL23FE} \\\\ Flights: 53 \\\\ Base: 28.0/h \\\\ New: 28.0/h \textcolor{gray}{$=$}};
  \node[box] (reg25) [below=of reg20] {\textbf{ESOS179} \\\\ Flights: 18 \\\\ Base: 9.3/h \\\\ New: 9.0/h \textcolor{red}{$\downarrow$}};
  \node[box] (reg26) [right=of reg25] {\textbf{EDUWUR24} \\\\ Flights: 8 \\\\ Base: 5.3/h \\\\ New: 3.0/h \textcolor{red}{$\downarrow$}};
  \node[box] (reg27) [right=of reg26] {\textbf{LFRGA} \\\\ Flights: 10 \\\\ Base: 8.0/h \\\\ New: 4.0/h \textcolor{red}{$\downarrow$}};
  \node[box] (reg28) [right=of reg27] {\textbf{EDWDBDS} \\\\ Flights: 15 \\\\ Base: 28.0/h \\\\ New: 27.0/h \textcolor{red}{$\downarrow$}};
  \node[box] (reg29) [right=of reg28] {\textbf{EISNALL} \\\\ Flights: 81 \\\\ Base: 16.0/h \\\\ New: 13.0/h \textcolor{red}{$\downarrow$}};
  \node[box] (reg30) [below=of reg25] {\textbf{EDUNTM44} \\\\ Flights: 29 \\\\ Base: 6.0/h \\\\ New: 6.0/h \textcolor{gray}{$=$}};
  \node[box] (reg31) [right=of reg30] {\textbf{LFBP1} \\\\ Flights: 43 \\\\ Base: 48.0/h \\\\ New: 46.0/h \textcolor{red}{$\downarrow$}};
  \node[box] (reg32) [right=of reg31] {\textbf{EDUL234} \\\\ Flights: 43 \\\\ Base: 20.0/h \\\\ New: 17.0/h \textcolor{red}{$\downarrow$}};
  \node[box] (reg33) [right=of reg32] {\textbf{LFRRMSI} \\\\ Flights: 27 \\\\ Base: 24.0/h \\\\ New: 23.0/h \textcolor{red}{$\downarrow$}};
  \node[box] (reg34) [right=of reg33] {\textbf{LECBBSU} \\\\ Flights: 6 \\\\ Base: 12.0/h \\\\ New: 6.0/h \textcolor{red}{$\downarrow$}};
  \node[box] (reg35) [below=of reg30] {\textbf{LFMLOLE} \\\\ Flights: 35 \\\\ Base: 4.0/h \\\\ New: 2.0/h \textcolor{red}{$\downarrow$}};
  \node[box] (reg36) [right=of reg35] {\textbf{ENO34} \\\\ Flights: 19 \\\\ Base: 7.0/h \\\\ New: 6.0/h \textcolor{red}{$\downarrow$}};
  \node[box] (reg37) [right=of reg36] {\textbf{EGCLW} \\\\ Flights: 37 \\\\ Base: 12.0/h \\\\ New: 7.0/h \textcolor{red}{$\downarrow$}};
  \node[box] (reg38) [right=of reg37] {\textbf{LECMZGZ} \\\\ Flights: 21 \\\\ Base: 4.0/h \\\\ New: 2.0/h \textcolor{red}{$\downarrow$}};
  \node[box] (reg39) [right=of reg38] {\textbf{MASD5HH} \\\\ Flights: 4 \\\\ Base: 8.0/h \\\\ New: 6.0/h \textcolor{red}{$\downarrow$}};
  \node[box] (reg40) [below=of reg35] {\textbf{MASD3WL} \\\\ Flights: 28 \\\\ Base: 8.0/h \\\\ New: 7.0/h \textcolor{red}{$\downarrow$}};
  \node[box] (reg41) [right=of reg40] {\textbf{LFBL2} \\\\ Flights: 5 \\\\ Base: 4.0/h \\\\ New: 2.0/h \textcolor{red}{$\downarrow$}};

  \draw[arrow] (reg0.east) -- (reg1.west);
  \draw[arrow] (reg1.east) -- (reg2.west);
  \draw[arrow] (reg2.east) -- (reg3.west);
  \draw[arrow] (reg3.east) -- (reg4.west);
  \draw[arrow] (reg5.east) -- (reg6.west);
  \draw[arrow] (reg6.east) -- (reg7.west);
  \draw[arrow] (reg7.east) -- (reg8.west);
  \draw[arrow] (reg8.east) -- (reg9.west);
  \draw[arrow] (reg10.east) -- (reg11.west);
  \draw[arrow] (reg11.east) -- (reg12.west);
  \draw[arrow] (reg12.east) -- (reg13.west);
  \draw[arrow] (reg13.east) -- (reg14.west);
  \draw[arrow] (reg15.east) -- (reg16.west);
  \draw[arrow] (reg16.east) -- (reg17.west);
  \draw[arrow] (reg17.east) -- (reg18.west);
  \draw[arrow] (reg18.east) -- (reg19.west);
  \draw[arrow] (reg20.east) -- (reg21.west);
  \draw[arrow] (reg21.east) -- (reg22.west);
  \draw[arrow] (reg22.east) -- (reg23.west);
  \draw[arrow] (reg23.east) -- (reg24.west);
  \draw[arrow] (reg25.east) -- (reg26.west);
  \draw[arrow] (reg26.east) -- (reg27.west);
  \draw[arrow] (reg27.east) -- (reg28.west);
  \draw[arrow] (reg28.east) -- (reg29.west);
  \draw[arrow] (reg30.east) -- (reg31.west);
  \draw[arrow] (reg31.east) -- (reg32.west);
  \draw[arrow] (reg32.east) -- (reg33.west);
  \draw[arrow] (reg33.east) -- (reg34.west);
  \draw[arrow] (reg35.east) -- (reg36.west);
  \draw[arrow] (reg36.east) -- (reg37.west);
  \draw[arrow] (reg37.east) -- (reg38.west);
  \draw[arrow] (reg38.east) -- (reg39.west);
  \draw[arrow] (reg40.east) -- (reg41.west);
\end{tikzpicture}
\caption{Regulation Solution Trajectory for the 3h case. Each regulation targets a flow like in Figure \ref{fig:flow_x_sample}.}
\label{fig:regulation_trajectory}
\end{figure}
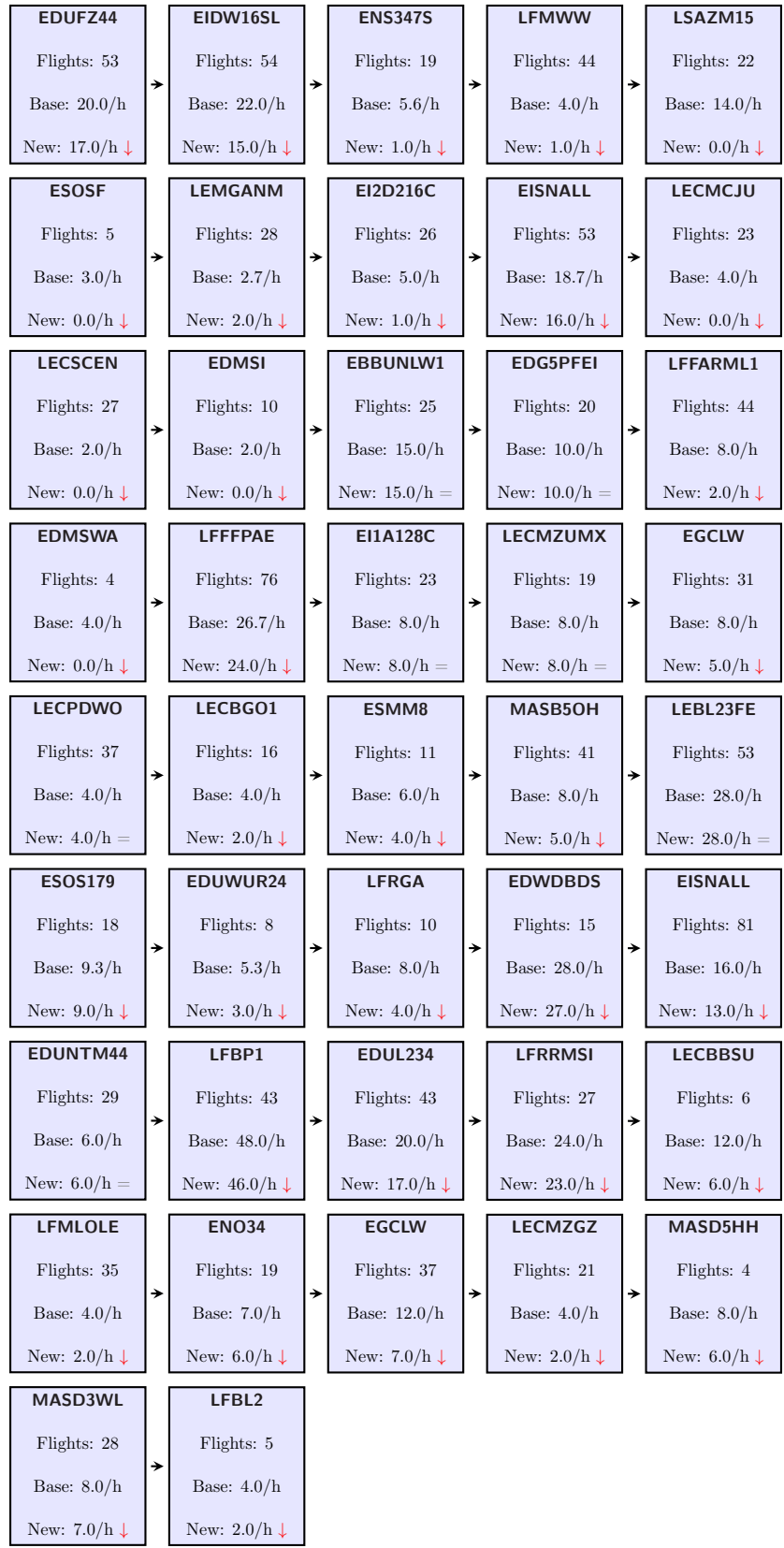

\section*{Data Protection}
The processing pipeline is fully secured with all processing conducted on computer systems belonging to data centers located within the European Union (Romania and Czech Republic).% The authors did not go to China, or brought data to China without notice.

\section*{Acknowledgments}
This work received funding from the SESAR DeepFlow Project [Grant Agreement ...]. The contents reflect only the authors' views. We are grateful for the discussions with EUROCONTROL Network Manager operators.

\bibliographystyle{ieeetr}
\bibliography{main}

\end{document}